\UseRawInputEncoding

\documentclass[12pt]{article}

\newcommand{\piant}{\Phi_{*}\frac{\partial }{\partial t }}
\newcommand{\pians}{\Phi_{*}\frac{\partial }{\partial s }}

\newcommand{\yuankuo}[1]{\left( #1\right) }
\newcommand{\jiankuo}[1]{\left\langle  #1\right\rangle }

\newcommand{\covt}[1]{\tilde{\nabla}_{#1}}

\newcommand{\FF}{F^{\prime\prime}\yuankuo{\frac{|du_{s,t}|^2}{2}}}

\newcommand{\jifen}[1]{\int_M #1 \mathrm{d}v_{g}}
\usepackage{color}
\usepackage{amssymb,latexsym}

\usepackage{pdfsync}
\usepackage[colorlinks, backref=page
]{hyperref}
\usepackage{amsmath,amsthm}
\usepackage{indentfirst}
\theoremstyle{plain}

\newtheorem{thm}{Theorem}[section]
\newtheorem{cor}{Corollary}[section]
\newtheorem{lem}{Lemma}

\theoremstyle{definition}
\newtheorem{defn}{Definition}[section]

\newtheorem{rem}{Remark}
\numberwithin{equation}{section}
\allowdisplaybreaks
\def \d {\mathrm{d}}

\def \Vol{\mathrm{Vol}}

\title{Stability for $ F $ harmonic map with two form and potential versus Stability for $ F $ symphonic  map  with  potential}
\author{Xiangzhi Cao\thanks{School of Information Engineering, Nanjing Xiaozhuang University, Nanjing 211171, China}\thanks{This work is surported by General Project of Basic Science (Natural Science) Research of  Universities and Colleges  in Jiangsu province (Grant No.
		22KJD110004)}\footnote{Email:aaa7756kijlp@163.com}}

\begin{document}
		\maketitle
		\tableofcontents
		\begin{abstract}
			
			In this paper,  we consider the stability of $ F $-harmonic map with $ m $-form and potential into pinched manifold. We also consider the stability of $ F $-symphonic  map with potential form or into compact $\Phi$-SSU manifold. We also consider the stability of $ F $-symphonic  map with potential into pinched manifold.
			

		\end{abstract}
		
		{\small
			\noindent{\it Keywords and phrases}: Stability; $F$ harmonic map with two form and potential; $F $ symphonic  map with potential . 
			
			\noindent {\it MSC 2010}: 58E15; 58E20 ; 53C27
		}
		
		\section{Introduction }		
	Harmonic map is an important topic in differential geomeotry. For the research of stability on harmonic map, Xin\cite{xin1980some} proved the nonexistence of nonconstant stable harmonic map from sphere into compact manifold.   Leung\cite{leung1982stability} proved the nonexistence of nonconstant stable harmonic map from compact manifold into  sphere. In \cite{Howard1985}, Howard proved the nonexistence of stable harmonic map from compact Riemannian manifold to $ \delta(n) $ pinched Riemannian manifold.

	Let $ M $ and $  N $ be two Riemannian manifolds.	For $ u: (M,g)\to (N,h) $, 
		$$E(u)=\int_{M}F\left(\frac{|du|^{2}}{2}\right)\d v_{g},$$
		whose critical point is called $ F $ harmonic map. It is well known that  $ p $ harmonic map and expotential harmonic map is the special case of $ F $ harmonic map.
		
		 For the research on $ F $ harmonic map, monotonicity formula and stability problem are two important topic in this field. They are all related to Liouville type theorem for $ F $ harmonic map. Dong \cite{dong2016liouville}  proved several Liouville theorems for $F$-harmonic maps from some complete Riemannian manifolds by assuming some conditions on the Hessian of the 	distance function, the degrees of $ F(t) $ and the asymptotic behavior of the map at infinity. Liu\cite{MR2259738} obtained the Liouville type theorem for stable $F$-harmonic map and   improved the results in \cite{ara2001stability}.	 Liu \cite{Liu2005} obtained Liouville theorem for F harmonic map on noncompact manifolds via conservation law. Kassi \cite{kassi2006liouville} also discussed  Liouville theorem for $ F $ harmonic map  via conservation law. 
		
		Ara \cite{Mitsunori1999Geometry} introduced the notion $ F $ harmonic map and discussed the stability of $ F $ harmonic map from  compact Riemannian manifold to $ \mathbb{S}^n $.
		 In \cite{ara2001stability}, Ara  proved that every stable $ F $-harmonic map into sufficently pinched simply-connected Riemannian manifold is constant. In \cite{ara2001instability}, Ara studied the unstability and nonexistence
		of $ F $-harmonic maps and introduced the notion of $ F $-strongly unstable
		and $ F $-unstable manifolds and discuss properties of such manifolds. The stability of $ \alpha $-harmonic map was studied  in \cite{torbaghan2022stability}.

			Koh \cite{Koh} studied the existence of magnetic geodesic in Riemannian manifold using heat flow method.  Branding\cite{Br} introduced harmonic map with two form and potential and the existence problem is solved provided  the curvature of the targeted manifold satisfying some conditions. In \cite{zbMATH07375745}, we generalized the main result  in \cite{Br}.

		In this paper, in order to generalize harmonic map with two form and potential,  we introduce the following functional
	\begin{defn} For $ u: (M,g)\to (N,h) $,
		$$E(u)=\int_{M}F\left(\frac{|du|^{2}}{2}\right)+u^* B+H(u)\d v_{g},$$
		whose critical point is called $ F $ harmonic map with $ m $ form and potential. In this paper, we will prove its	Euler-Lagrange equation is 
			\begin{equation}
			\delta^\nabla\left( F^{\prime}\left(\frac{|du|^{2}}{2}\right)du\right) - Z_{x}\left(du(e_1)\wedge \cdots \wedge du(e_m)\right)-\nabla H(u)=0.
		\end{equation}
	where $\{e_1,e_2,\cdots, e_m\} $ is the orthonormal basis of $M^m$.
	\end{defn}	
	
	In this paper, we will establish Howard type result and Okayasu type result for $ F $-harmonic with two form and potential. We will use the method  to study the stability of   $ F $-harmonic with two form and potential.

In  \cite{nakauchi2011variational},  Nackauchi et al. 
 defined symphonic map wich is the critical point of 
\begin{equation*}
	\begin{split}
		E(u)=\int_M \frac{|u^{*}h|^2}{4}\d v_g. 
	\end{split}
\end{equation*} 
In	\cite{kawai2011some}, Kawai and  Nakauchi  studied the stability of symphonic map  from or into $ \mathbb{S}^{n}. $ In \cite{zbMATH07069207}, Kawai and Nakauchi  obtaind Liouvill theorem for symphonic map in terms of the curvature conditions of the domain manifold or target manifold. In \cite{Misawa2023}, Misawa and Nakauchi gived finite-time blow-up phenomenon of rotationally symmetric solutions to the symphonic map flow. In \cite{Nakauchi2022}, Nakauchi obtained Liouville theorem for symphonic map in terms of $ m $-symphonic energy. In \cite{Nakauchi2019}, Nakauchi gived the stress energy tensor of symphonic map(see \cite{nakauchi2011variational}).

			\begin{defn}For $ u: (M,g)\to (N,h) $,
			$$\Phi_{sym}(u)=\int_{M}F\left(\frac{|u^*h|^{2}}{4}\right)+H(u)\d v_g,$$
			whose critical point is called $ F $-symphonic map with potential .	Euler-Lagrange equation is 
			\begin{equation}
				\delta^\nabla\left( (F^{\prime}\left(\frac{|u^*h|^{2}}{2}\right)\sigma_{u} \right)-\nabla H(u)=0,
			\end{equation}
		where $ \sigma_{u}(\cdot)=h(du(\cdot),du(e_i))du(e_i). $ When $ H=0 ,F(x)=x$, $ F $-symphonic map is also called $ \Phi $-harmonic map(see \cite{Han2019HarmonicMA}) or symphonic map.
		\end{defn}

 When $ H=0, $ in \cite{han2014monotonicity}, Han et al.  studied the stability of $ F $-symphonic map in the case where	the domain  manifold is $ S^n $ and target manifold is sphere. In \cite{Li2017NonexistenceOS}, Li et al. also studied the case where the domain manifold is  compact convex hypersurface and the target manifold is convex surface .

	In 	\cite{Howard1986},     Howard and Weil studied strongly unstable manifold  for submanifold in Euclidean space. Strongly unstable manifold is related to Liouville type theorem for stable harmonic map. In \cite{ara2001instability}, Ara studied $ SSU $ manifold.
	
	\begin{defn}[c.f. \cite{Han213213123123}\cite{Han2019HarmonicMA} ]
		A Riemannian manifold $M^{m}$ is said to be $\Phi$-superstrongly unstable ($\Phi$-SSU) if there exists an isometric immersion of $M^{m}$ in $R^{m+p}, p>0$ with its second fundamental form $B$ such that for all unit tangent vectors $X$ to $M^{m}$ at every point $p \in M$, the following inequality holds:
				\[
		\left\langle Q_{x}^{M}(X), X\right\rangle_{M}=\sum_{i=1}^{m}\bigg(4\left\langle B\left(X, e_{i}\right), B\left(X, e_{i}\right)\right\rangle-\left\langle B(X, X), B\left(e_{i}, e_{i}\right)\right\rangle\bigg)<0,
		\]
		where $\left\{e_{i}\right\}_{i=1}^{m}$ is the local orthogonal frame field near $y$ of $M$.		
	\end{defn}
\begin{rem}
	The dimension of any compact $ \Phi $-SSU manifolds $ N $ is
	greater than 4, see \cite[Theorem 7.1]{Han2019HarmonicMA}.
\end{rem}
	\begin{rem}
		$ \Phi$-$ SSU  $ manifold was studied in  (\cite{Han213213123123}\cite{Han2019HarmonicMA} ).  $\Phi$-SSU manifold was introduced by Han and Wei in \cite{Han2019HarmonicMA} and some interesting examples of $\Phi$-$S S U$ manifolds were studied.
	\end{rem}
	\begin{rem}
Here, we list some examples of $ \Phi $-SSU manifold,

		(1) 	By \cite[Corollary 5.1]{Han2019HarmonicMA}, the graph of $ f(x)=x_1^2+\cdots+x_n^2,x=(x_1,\cdots,x_n)\in \mathbb{R}^{n+1} $ is $ \Phi $-SSU manifold if and only if $ n>4. $
		
		(2)	By \cite[Corollary 5.2]{Han2019HarmonicMA},The standard sphere$  S^n $ is $ \Phi $-SSU  if and only if $ n > 4 $.
		
		(3) By Theorem 5.1 in \cite{Han2019HarmonicMA} . Let $N$ be a hypersurface in Euclidean space. Then $N$ is $\Phi$-SSU if and only if its principal curvatures satisfy
		\[
		0<\lambda_{1} \leq \lambda_{2} \leq \cdots \leq \lambda_{n}<\frac{1}{3}\left(\lambda_{1}+\cdots+\lambda_{n-1}\right)
		\]
		
		(4) \cite[Theorem 5.2]{Han2019HarmonicMA}. Let $\tilde{N}$ be a compact convex hypersurface of $\mathbb{R}^{q}$ and $N$ be a compact connected minimal $k$-submanifold of $\widetilde{N}$. Assuming that the principal curvatures $\lambda_{i}$ of $\tilde{N}$ satisfy $0<\lambda_{1} \leq \cdots \leq \lambda_{q-1}$. If $\left\langle\operatorname{Ric}^{N}(\mathrm{x}), \mathrm{x}\right\rangle>\frac{3}{4} k \lambda_{q-1}^{2}$ for any unit tangent vector $\mathrm{x}$ to $N$, then $N$ is $\Phi-\mathrm{SSU}$.
		
		(5) \cite[Theorem 5.3]{Han2019HarmonicMA}. A compact minimal $k$-submanifold $N$ of an ellipsoid $E^{q-1}$ in $\mathbb{R}^{q}$ with $\left\langle\operatorname{Ric}^{N}(\mathrm{x}), \mathrm{x}\right\rangle>\frac{3}{4} \frac{\left(\max _{1 \leq i \leq q}\left\{a_{i}\right\}\right)^{2}}{\left(\min _{1 \leq i \leq q}\left\{a_{i}\right\}\right)^{4}} k$ for any unit tangent vector $\mathrm{x}$ of $N$ is $\Phi$-SSU.
	
	(6)	\cite[Corollary 5.4]{Han2019HarmonicMA}. A compact minimal $k$-submanifold $N$ of the unit sphere $S^{q-1}$ with $\left\langle\operatorname{Ric}^{N}(\mathrm{x}), \mathrm{x}\right\rangle>\frac{3}{4} k$ for any unit tangent vector $\mathrm{x}$ to $N$ is $\Phi$-SSU. Thus $N$ satisfies Theorem $1.1(a),(b),(c)$, and $(d)$.
	
	(7)\cite[Theorem 5.4]{Han2019HarmonicMA}. Let $N$ be a compact $k$-submanifold of the unit sphere $S^{q-1}$ in $\mathbb{R}^{q}, k>4$, and Let $\mathrm{B}_{1}$ be the second fundamental form of $N$ in $S^{q-1}$. If
	\[
	\left\|\mathrm{B}_{1}\right\|^{2}<\frac{k-4}{\sqrt{k}+4},
	\]
	then $N$ is $\Phi$-SSU.
	
	(8) \cite[Theorem 9.1]{Han2019HarmonicMA}. Let $N=G / H$ be a compact irreducible homogeneous space of dimension $n$ with first eigenvalues $\lambda_{1}$ and scalar curvature Scal $^{N}$. Then
 $N$ is $\Phi$-SSU  if and only if 	 $\lambda_{1}<\frac{4}{3 n}$ Scal $^{N}$.
	\end{rem}

In section  \ref{sec4} of  this paper, we  will prove  the stability of  $ F $-symphonic map with potential  from or into compact $\Phi$-SSU manifold.  We will also get the Howard type result or Okayasu type result for $ F $-symphonic map with potential  into $ \delta $-pinched manifold. We will use the method in \cite{torbaghan2022stability} to study the stability of  $ F $-symphonic map with potential .

	Throughout this paper,  $c_{F}:=\inf \left\{c \geq 0 |\frac{F^{\prime}(t)}{t^{c}} \quad \text{is nonincreasing} \right\}$, thus we have  $ F^{\prime\prime}(x)x \leq  c_F F^{\prime}(x).$ This inequality will be used  often hereafter.
		
		\section{$F$ harmonic map with $ m $ form and potential when  targeted manifold is $ \delta $-pinched }\label{sec2}
		
		Let $ (N,h) $ be a Riemannian manifold. 	Now, we  define a $(k+1)$-form $\Omega \in \Gamma\left(\Lambda^{k+1} T^{*} N\right)$ a smooth vector bundle homomorphism $Z: \Lambda^{k} TN \rightarrow TN$, by the equation
		$$
		\left\langle\eta, Z_{x}\left(\xi_{1} \wedge \cdots \wedge \xi_{k}\right)\right\rangle_h=\Omega_{x}\left(\eta, \xi_{1}, \ldots, \xi_{k}\right), dB=\Omega,
		$$
		for all $x \in M$ and all $\eta, \xi_{1}, \ldots, \xi_{k} \in T_{x} N$, $B \in \Gamma\left(\Lambda^{k} T^{*} N\right)$. 
		\begin{lem}[c.f.\cite{koh2008evolution}]	
			Let $ u_t:(M^m,g)\to (N^n,h) $ be  a smooth deformation of u such that such that $ u_0=u, v=\frac{\partial u_t}{\partial t}|_{t=0} $  and	$B \in \Gamma\left(\Lambda^{m} T^{*} M\right),$		
		\begin{equation}\label{}
			\begin{split}
				\left.\int_{\Sigma} \frac{d}{d t}\right|_{t=0} \varphi_{t}^{*} B=
							&\int_{\Sigma} \Omega\left(v,(d \varphi)^{\underline{k}}\left(\mathrm{vol}_{g}^{\sharp}\right)\right) \d v_{g},
			\end{split}
		\end{equation}
	where $ (d u)^{\underline{m}}\left(\mathrm{vol}_{g}^{\sharp}\right):=du(e_1)\wedge du(e_2)\cdots du(e_m) .$	
	\end{lem}
	We also have  
		\begin{thm}[The first varation formula, c.f. ]
			
			Let $ u_t:(M^m,g)\to (N^n,h) $ be  a smooth deformation of u such that such that $ u_0=u, V=\frac{\partial u_t}{\partial t}|_{t=0} $  and	$B \in \Gamma\left(\Lambda^{m} T^{*} M\right),$	
			\begin{equation*}
				\begin{split}
					&\frac{d}{dt}\int_{M}F\left(\frac{|du|^{2}}{2}\right)+u^* B+H(u)\d v_g|_{t=0}\\
					=&-\jifen{\jiankuo{	\delta^\nabla\left( (F^{\prime}\left(\frac{|du|^{2}}{2}\right)du\right) )+ Z \left(du(e_1) \wedge \cdots \wedge du(e_m)\right)+\nabla H(u), v}}.	
				\end{split}
			\end{equation*}
		\end{thm}

	Hereafter, in this section, we consider the case where $ B $ is two form and $ M $ is Riemannian surface.
		
		\begin{thm}[the second variation formula ]
				Let $ u_{s,t}:(M^m,g)\to (N^n,h) $ be  a smooth deformation of u such that such that $ u_0=u, V=\frac{\partial u_t}{\partial t}|_{t=0} ,W=\frac{\partial u_t}{\partial s}|_{t=0}$  and	$B \in \Gamma\left(\Lambda^{m} T^{*} M\right),$ then 
		\begin{equation}\label{eq2}
			\begin{split}
				\left.	\frac{\partial^{2}}{\partial s \partial t}\right|_{s, t=0}  E(u_{s,t})	=&\jifen{\FF\jiankuo{\tilde{\nabla}V,du}\jiankuo{{\tilde{\nabla}W,du}}}\\
				&+\int_{M} F^{\prime}(\frac{|d u|^{2}}{2}\left\{\langle\tilde{\nabla} V, \tilde{\nabla} W\rangle-\sum_{t=1}^{m} h( R^N\left(V, u_* e_{i}\right) u_* e_{i}, W)\right\} \d v_{g}\\
				&+ \jifen{\nabla^2 H(V,W)}+\int_{M} \jiankuo{V,Z(\covt{e_1}{W}\wedge du(e_2)\wedge \cdots \wedge du(e_m)))} \d v_{g}\\
				&+\int_{M} \jiankuo{V,Z\big(du(e_2))\wedge\covt{e_2}{W}\wedge \cdots \wedge du(e_m)\big) } \d v_{g}\\
				&+\cdots+ \int_{M} \jiankuo{V,Z\big(du(e_1))\wedge \cdots \wedge \covt{e_m}{W}\big) } \d v_{g}
			\end{split}
		\end{equation}
		\end{thm}
		\begin{proof}
				Let $ u_{s,t}:(M^m,g)\to (N^n,h) $ be  a smooth deformation of u such that such that $ u_0=u, V=\frac{\partial u_t}{\partial t}|_{t=0} $  and	$B \in \Gamma\left(\Lambda^{m} T^{*} M\right),$
			Let $ \Phi: (-\delta, \delta)\times (-\delta, \delta)\times M \to N $, defined by $ \Phi(s,t,x)=u_{s,t}(x). $
			It is not hard to see that 
				\begin{align}
			\frac{\partial^{2}}{\partial s \partial t}\int_{M}H(u)\d v_g= \int_{M}  \nabla^2 H(\piant,\pians) +\left\langle  \nabla_{\pians}\piant, \nabla H \right\rangle  \d v_{g},
			\end{align}	
and			
		\begin{equation*}
			\begin{split}
				\frac{\partial^{2}}{\partial s \partial t} \int_{M}u^* B\d v_g=& \int_{M} \jiankuo{\piant,Z(\covt{e_1}{\pians}\wedge d\Phi(e_2))}\d v_{g}\\
				&+\int_{M} \jiankuo{\piant,Z(d\Phi(e_2))\wedge\covt{e_2}{\pians} } \d v_{g}\\
				&+\int_{M} \jiankuo{\covt{\pians}{\piant},Z(du(e_1),du(e_2))}  \d v_g \\ 
			\end{split}
		\end{equation*}
			
	Using well known the second variation formula for $ F $ harmonic map, we get 		
		\begin{equation}\label{index1}
			\begin{split}
						\left.	\frac{\partial^{2}}{\partial s \partial t}\right|_{s, t=0}  E	=&\jifen{\FF\jiankuo{\tilde{\nabla}V,du}\jiankuo{{\tilde{\nabla}W,du}}}\\
				&+\int_{M} F^{\prime}(\frac{|d u|^{2}}{2}\left\{\langle\tilde{\nabla} V, \tilde{\nabla} W\rangle-\sum_{t=1}^{m} h( R^N\left(V, u_* e_{i}\right) u_* e_{i}, W)\right\} \d v_{g}\\
				&+ \jifen{\nabla^2 H(V,W)}+\int_{M} \jiankuo{V,Z(\covt{e_1}{W}\wedge du(e_2)\wedge \cdots \wedge du(e_m)))} \d v_{g}\\
				&+\int_{M} \jiankuo{V,Z\big(du(e_2))\wedge\covt{e_2}{W}\wedge \cdots \wedge du(e_m)\big) } \d v_{g}\\
				&+\cdots+ \int_{M} \jiankuo{V,Z\big(du(e_1))\wedge \cdots \wedge \covt{e_m}{W}\big) } \d v_{g}
			\end{split}
		\end{equation}
		\end{proof}
		\begin{defn} 
			For $F$-harmonic map with $ m $-form and potential, we define the index form $ I_{FBH}(V,W) $ be the RHS of \eqref{index1}.
				A map $ u $ is called $ FBH$-stable if $ I_{FBH}(V,V) $ is nonnegative for any nonzero vector filed $V$.  Otherwise, it is called  $ FBH$-unstable.
		\end{defn}
		\begin{thm}\label{thm2.3}
		Let $ (M,g)  $ be compact Riemannian surface, $ N $ be a  compact simply connected $ \delta $-pinched n-dimensional  Riemannian manifold. Assume that $ c_F<\frac{n}{2}-1 $ and 
			$$  \Phi_{n,F}(\delta)+2(n+1) (3\|\Omega\|_\infty+2\|\Omega\|_\infty\frac{n}{4}k_3^2(\delta)) <0, $$
			where $\Phi_{n,F}(\delta)= \left(2 c_{F}+1\right)\left\{\frac{n}{4} k_{3}^{2}(\delta)+k_{3}(\delta)+1\right\}-\frac{2 \delta}{1+\delta}(n-1).$ 
			Every $ FBH$-stable  $ F $-harmonic map with two form and potential  $ u : (M,g) \to (N,h)$  is constant.
		\end{thm}
	\begin{proof}
		
			As in \cite{MR2259738}, we assume the sectional curvature of N is equal to $ \frac{2\delta}{1+\delta} .$ Let  the vector  bundle $ E=TN\oplus \epsilon(N) ,$ here, $ \epsilon(N) $ is  the trivial bundle. As in\cite{ara2001stability} \cite{MR2259738}, we can define a   metric connection  $ \nabla^{\prime\prime} $ on $ E $ as follows:
		\begin{equation}\label{ccfd}
			\begin{split}
				&\nabla^{\prime\prime}_XY={}^{N}\nabla_XY-h(X,Y)e;\\
				&\nabla^{\prime\prime}_Xe=X,	
			\end{split}
		\end{equation}
		We cite the functions in \cite{ara2001stability} or \cite{MR2259738}, 
		\begin{equation*}
			\begin{split}
				k_{1}(\delta)=\frac{4(1-\delta)}{3 \delta}\left[1+\left(\sqrt{\delta} \sin \frac{1}{2} \pi \sqrt{\delta}\right)^{-1}\right],
				k_{2}(\delta)=\left[\frac{1}{2}(1+\delta)\right]^{-1}  k_{1}(\delta). \\
			\end{split}
		\end{equation*}	
		By \cite{ara2001stability} \cite{MR2259738}, we know that  there exists a flat connection $ \nabla^{\prime} $ such that 
		\begin{equation*}
			\begin{split}
				\|\nabla^{\prime}-\nabla^{\prime\prime}\|\leq \frac{1}{2}k_3(\delta),
			\end{split}
		\end{equation*}
		where $ k_3(\delta) $ is defined as 
		\begin{equation*}
			\begin{split}
				k_{3}(\delta)=k_{2}(\delta) \sqrt{1+\left(1-\frac{1}{24} \pi^{2}\left(k_{1}(\delta)\right)^{2}\right)^{-2}} .
			\end{split}
		\end{equation*}
		
		take a cross section $ W  $ of $ E $ and let $ W^T $ denotes the $ TN $ component of $ W $
		\begin{equation}\label{cjk}
			\begin{split}
				&	I\left(W^{T}\right.\left., W^{T}\right) \\
				=& \int_{M}   F{''}\left(\frac{|\mathrm{d} u|^{2}}{2}\right) \sum_{i=1}^{m}\left\langle\tilde{\nabla}_{e_{i}} W^{T}, u_{*} e_{i}\right\rangle^{2} \mathrm{d}v_{g} \\
				&	+\int_{M}  F{'}\left(\frac{|\mathrm{d} u|^{2}}{2}\right) \cdot \sum_{i=1}^{m}\left\{\left|\tilde{\nabla}_{e_{i}} W^{T}\right|^{2}-h\left(R^{N}\left(W^{T}, u_{*} e_{i}\right) u_{*} e_{i}, W^{T}\right)\right\} \mathrm{d}v_{g}	\\
				&+ \jifen{\nabla^2 H(W^{T},W^T)}+\int_{M} \jiankuo{W^{T},Z(\covt{e_1}{W^{T}}\wedge du(e_2))} \mathrm{d}v_{g}\\
				&+\int_{M} \jiankuo{W^{T},Z(du(e_1))\wedge\covt{e_2}{W} } \mathrm{d}v_{g}\\ 
				\leq &\int_{M} F{'}\left(\frac{|\mathrm{d} u|^{2}}{2}\right) \sum_{i=1}^{m} \left\{\left(2 c_{F}+1\right)\left|\tilde{\nabla}_{e_{i}} W^{T}\right|^{2}\right.
				\left.-h\left(R^{N}\left(W^{T}, u_{*} e_{i}\right) u_{*} e_{i}, W^{T}\right)\right\} \mathrm{d}v_{g}+I_1,
			\end{split}
		\end{equation}
			where 
	\begin{equation*}
		\begin{split}
			I_1(W^T,W^T)=&	\sum_{i} \jifen{\nabla^2 H(W^{T},W^T)}+\int_{M} \jiankuo{W^{T},Z(\covt{e_1}{W^{T}}\wedge du(e_2))} v_{g}\\
			&+\int_{M} \jiankuo{W^{T},Z(du(e_2))\wedge\covt{e_2}{W^T} } v_{g}.
		\end{split}
	\end{equation*}	
%
		Then, we have by \cite[(5.3)]{ara2001stability} 
		$$
		\begin{aligned}
			\sum_{i=1}^m\left|\tilde{\nabla}_{e_i} W^T\right|^2=& \sum_{i=1}^m\left|\left(\nabla_{u_{*} e_i}^{\prime \prime} W\right)^T\right|^2+\langle W, e\rangle^2|\mathrm{~d} u|^2 \\
			&-2 \sum_{i=1}^m\langle W, e\rangle\left\langle\nabla_{u_* e_i}^{\prime \prime} W^T, u_* e_i\right\rangle .
		\end{aligned}
		$$
		Since $N$ is $\delta$-pinched, by \cite[(5.4)]{ara2001stability}

		$$ h\left(R^N\left(W^T, u_* e_i\right) u_* e_i, W^T\right) \geq \frac{2 \delta}{1+\delta}\left\{\left|W^T\right|^2\left|u_* e_i\right|^2-\left\langle W^T, u_* e_i\right\rangle^2\right\}. $$
		Substituting () and () into (\ref{cjk}), we obtain
		$$
		I\left(W^T, W^T\right) \leq \int_M F^{\prime}\left(\frac{|\mathrm{d} u|^2}{2}\right) \cdot q(W) \d v_g+I_1.
		$$
		where
		\begin{equation*}
			\begin{split}
				q(W)=&\left(2c_{F}+1\right)\left\{\sum_{i=1}^{m}\left|\left(\nabla_{u_{*} e_{i}}^{\prime \prime} W\right)^{T}\right|^{2}+\langle W, e\rangle^{2}|\mathrm{~d} u|^{2}\right.\left.\quad-2 \sum_{i=1}^{m}\langle W, e\rangle\left\langle\nabla_{u_{*} e_{i}}^{\prime \prime} W^{T}, u_{*} e_{i}\right\rangle\right\} \\
				&-\frac{2 \delta}{1+\delta} \sum_{i=1}^{m}\left\{\left|W^{T}\right|^{2}\left|u_{*} e_{i}\right|^{2}-\left\langle W^{T}, u_{*} e_{i}\right\rangle^{2}\right\} .
			\end{split}
		\end{equation*}
	
	Let $\mathcal{W}:=\left\{W \in \Gamma(E) ; \nabla^{\prime} W=0\right\}$, then $\mathcal{W}$ with natural inner product is isomorphic to $\mathbf{R}^{n+1}$. Define a quadratic form $Q$ on $\mathcal{W}$ by 	
	$$ Q(W):=\int_{M} F^{\prime}\left(\frac{|\mathrm{d} u|^{2}}{2}\right) \cdot q(W)  \d v_g +I_1(W,W).$$ 
	
	Taking an orthonormal basis $\left\{W_{1}, W_{2}, \ldots, W_{n}, W_{n+1}\right\}$ of $\mathcal{W}$ with respect to its natural inner product such that $ W_1,W_2, \cdots, W_n $ is tangent to $N$, we obtain

	By Liu\cite{MR2259738},	
		\begin{equation*}
			\begin{split}
				\operatorname{Tr}_g I \leq \int_{M} F^{\prime}\left(\frac{|\mathrm{d} u|^2}{2}\right)\Phi_{n,F}(\delta)|du |^2 \d v_g+ \operatorname{Tr}_g (I_1(W_i^T,W_i^T)).
			\end{split}
		\end{equation*}
	By the relation between $ B $  and $ \Omega $,  we have 
		\begin{equation*}
			\begin{split}
				\jiankuo{W_i^{T},Z(du(e_2))\wedge\covt{e_2}{W_i} } &=\Omega (W_i^T,du(e_1),\covt{e_2}{W_i^T})\\
				&\leq \|\Omega\|_\infty |W_i^T|du(e_1)||\left(\nabla_{u_* e_2}^{\prime \prime} W\right)^T-\langle W_i^T, e\rangle u_* e_2|\\
				&\leq   \|\Omega\|_\infty \bigg(|du(e_1)|^2+2|\left(\nabla_{u_* e_2}^{\prime \prime} W_i\right)^T |^2+2|du(e_2) |^2\bigg).
			\end{split}
		\end{equation*}
	However,	
		\begin{equation*}
			\begin{split}
				\sum_{i=1}^{n+1}	|\left(\nabla_{u_* e_2}^{\prime \prime} W_i\right)^T |^2=&	\sum_{i=1}^{n+1}\sum_{j=1}^{n}	| \left \langle \left(\nabla_{u_* e_2}^{\prime \prime} W_i\right), W_j \right \rangle  |^2\\
				=&\sum_{i=1}^{n+1}\sum_{j=1}^{n}	| \left \langle \left(\nabla_{u_* e_2}^{\prime \prime} W_j\right), W_i \right \rangle  |^2\\
				=&\sum_{j=1}^{n}	|  \nabla_{u_* e_2}^{\prime \prime} W_j |^2\leq \frac{n}{4}k_3^2(\delta)|du(e_2) |^2.
			\end{split}
		\end{equation*}
		Hence, we get 
		\begin{equation*}
			\begin{split}
				&\int_{M} \jiankuo{W_i^{T},Z(\covt{e_1}{W_i^{T}}\wedge du(e_2))}d v_{g}
				+\int_{M} \jiankuo{W_i^{T},Z(du(e_1))\wedge\covt{e_2}{W_i^T} } v_{g}\\
				\leq&  \int_M \|\Omega\|_\infty \bigg(3|du(e_1)|^2+3|du(e_2) |^2\bigg)+2\|\Omega\|_\infty\frac{n}{4}k_3^2(\delta)\left( |du(e_2) |^2 +|du(e_1) |^2\right) \d v_g\\
				& \leq \int_{M} (3\|\Omega\|_\infty+2\|\Omega\|_\infty\frac{n}{4}k_3^2(\delta)) |du |^2 \d v_g.
			\end{split}
		\end{equation*}
So, we get

	\begin{equation*}
	\begin{split}
		\operatorname{Tr}_g I \leq \int_{M} F^{\prime}\left(\frac{|\mathrm{d} u|^2}{2}\right)\Phi_{n,F}(\delta)|du |^2 \d v_g+ 2(n+1)\int_{M} (3\|\Omega\|_\infty+2\|\Omega\|_\infty\frac{n}{4}k_3^2(\delta)) |du |^2 \d v_g.
	\end{split}
\end{equation*}	

\end{proof}

	Next, we give Howard type theorem 
		
		\begin{thm}
			Let $F:[0, \infty) \rightarrow[0, \infty)$ be a $C^{2}$ strictly increasing function.  Let $ M $ be a  compact Riemannian manifold . Let $N$ be a compact simply-connected $\delta$-pinched $n$-dimensional Riemannian manifold. Assume that $n$ and $\delta$ satisfy $c_F<\frac{n-2-2m}{4}$ and
		\begin{equation*}
			\begin{split}
				\Psi_{n, F}(\delta):=&\int_{0}^{\pi}\bigg\{\left(2 c_{F}+1+2\|\Omega\|_\infty \sqrt{g_2(\rho,\delta)}\right) g_{2}(t, \delta)\left(\frac{\sin \sqrt{\delta} t}{\sqrt{\delta}}\right)^{n-1}\\
				&-(n-1) \delta \cos ^{2}(t) \sin ^{n-1}(t)\bigg\} \d t<0,
			\end{split}
		\end{equation*}
			where $g_{2}(t, \delta)=\max \left\{\cos ^{2}(t), \delta \sin ^{2}(t) \cot ^{2}(\sqrt{\delta} t)\right\}$. Then  every stable $ F $-harmonic map with two form and potential   $u: (M,g) \rightarrow (N,h)$ must be  constant.
		\end{thm}
		\begin{rem}
			The conditon is different from that in Ara\cite{Mitsunori1999Geometry} and Liu \cite{MR2259738}.
		\end{rem}
	\begin{proof}
		We modify the proof in \cite{ara2001stability}. For $ y\in N, $ let $ V^{y}=\nabla f\circ \rho_y,$  we can approximate $ f $ by $ f_k $, one can refer to \cite{ara2001stability} for the construction of $ f_k $.  Let   $ V_k^{y}=\nabla f_k\circ \rho_y  $. Then $ V_k^y $ converges uniformly to $ V^y $ as $ k\to \infty. $ By the second variational formula for $F$ harmonic map with two form and potential, 
		\begin{equation*}
			\begin{split}
				&\lim\limits_{k\to \infty}	\int_N  I(V_{k}^{y},V_{k}^{y})dv_N(y)\\
				\leq & \int_N	\int_{M} F^{\prime}\left(\frac{\left\|du\right\|^{2}}{4}\right) \sum_{i}^{m}   \bigg[(2c_F+1)|\tilde{\nabla}_{e_i} V_k^y|^2+h\left( R^{N}\left(V_k^y, d u\left(e_{i}\right)\right) V_k^y, d u\left(e_{i}\right)\right)\\
				&+I_1(V_{k}^{y},V_{k}^{y})\bigg] d v_{g} dv_N(y),
		\end{split}
		\end{equation*}
		where 
	\begin{equation*}
		\begin{split}
			I_1(V_{k}^{y},V_{k}^{y})=&	\sum_{i} \jifen{\nabla^2 H(V_{k}^{y},V_{k}^{y})}+\int_{M} \jiankuo{V_{k}^{y},Z(\covt{e_1}{V_{k}^{y}}\wedge du(e_2))}\d v_{g}\\
			&+\int_{M} \jiankuo{V_{k}^{y},Z(du(e_1))\wedge\covt{e_2}{V_{k}^{y}} } \d v_{g}. 
		\end{split}
	\end{equation*}	
Notice that 
\begin{equation*}
	\begin{split}
		| \tilde{\nabla}_{e_i}V_{k}^{y}|^2 \leq g_2(\rho,\delta)|u_{*}(e_i) |^2.
	\end{split}
\end{equation*}
Thus, we have 
\begin{equation*}
	\begin{split}
		I_1 \leq \int_M 2\|\Omega\|_\infty \sqrt{g_2(\rho,\delta)}|du(e_1)||du(e_2)| \d v_g\\
		\leq \int_M 2\|\Omega\|_\infty \sqrt{g_2(\rho,\delta)}|du |^2 \d v_g.
	\end{split}
\end{equation*}
By \cite[(3.3),(3.4)]{ara2001stability},
\begin{equation}\label{ck}
	\d v_N(y)\leq \Vol(S^{n-1}) \int_{0}^{\pi}\left(\frac{\sin(\sqrt{\delta}t)}{\sqrt{\delta}} \right)^{n-1}\d t\
\end{equation}
and
\begin{equation}\label{ckk}
	\begin{split}
			\int_{N} & \left\langle R^{N}\left(V^{y}, X\right) X, V^{y}\right\rangle v_{N}(y) \\
		& \geq(n-1)|X|^{2} \operatorname{Vol}\left(S^{n-1}\right) \int_{0}^{\pi} \delta \cos ^{2}(\rho) \sin ^{n-1}(\rho) d \rho .
	\end{split}
\end{equation}

	Thus, we have 
		\begin{equation*}
			\begin{split}
				&\lim\limits_{k\to \infty}	\int_N  I(V_{k}^{y},V_{k}^{y})dv_N(y)\\
				\leq & 	\int_{M} F^{\prime}\left(\frac{\left\|u^{*} h\right\|^{2}}{4}\right)\bigg( \bigg[\int_N  (2c_F+1)g_2(\rho,\delta) |du |^2dv_N(y)\\
				&\quad\quad\quad\quad\quad\quad\quad\quad+\sum_{i=1}^{m} \int_Nh\left( R^{N}\left(V^y, d u\left(e_{i}\right)\right) V^y, d u\left(e_{i}\right)\right)dv_N(y)\bigg]\\
				&+\int_N2\|\Omega\|_\infty \sqrt{g_2(\rho,\delta)}|du|^2 dv_N(y)\bigg)d v_{g} \\
				\leq & 	\int_{M} F^{\prime} \bigg[  (2c_F+1)g_2(\rho,\delta) |du |^2\Vol(S^{n-1}) \int_{0}^{\pi}\left(\frac{\sin(\sqrt{\delta}t)}{\sqrt{\delta}} \right)^{n-1}\d t\\
				& -(n-1)|\d u |^2Vol(S^{n-1}) \int_{0}^{\pi}\delta\cos^2(\rho)\sin^{n-1}(\rho)d\rho\bigg]\\
				&+2\|\Omega\|_\infty \sqrt{g_2(\rho,\delta)}|\d u |^2\Vol(S^{n-1})\int_{0}^{\pi} \left(\frac{\sin(\sqrt{\delta}t)}{\sqrt{\delta}} \right)^{n-1} \d t \bigg)\d v_{g} ,
			\end{split}
		\end{equation*}
	 the theorem follows.

\end{proof}

		\begin{thm}\label{thm2.2} 	Let $  u:(M^{n-1}, g) \to (N^n, h) $ be $F$-harmonic map with $ n-1 $ form and potential.
			Let $  u:(M^{n-1}, g) \to (N^n, h) $ be a
			totally geodesic isometric immersion. Then $u$ is FBH-unstable if the Ricci curvature
			of  $ N^n $ is positive.
		\end{thm}
		
		\begin{proof}
		Here we modify the proof of \cite[Theorem 10]{torbaghan2022stability}. Let $ V $ be a unit normal vector field of $ M^{n-1}  $ in 	$ N^n $. Since $ u $ is the totally geodesic, then  by \cite[Theorem 11]{torbaghan2022stability}			
			\begin{equation*}
				\begin{split}
					\tilde{\nabla} V=0.
				\end{split}
			\end{equation*}
			So, by \eqref{eq2}, we have 
			\begin{equation*}
				\begin{split}
					I(V,V)=-\int_{M}  F^{\prime}(\frac{|d u|^{2}}{2}) Ric(V,V)\d v_g.
				\end{split}
			\end{equation*}
		
	\end{proof}
	\section{$ F $ harmonic map}
		Motivated by the method in \cite{torbaghan2022stability}, we can establish  
	\begin{thm}\label{thm2.1}
		Let $u:\left(M^{m}, g\right) \longrightarrow\left(N^{n}, h\right)$ be a nonconstant $F$-harmonic map with potential between Riemannian manifolds. Suppose that $ \nabla^2H $ is semi-negative, 
		\[
		4c_F h+\theta(v)<\operatorname{Ric}(v, v),
		\]
		at each $x \in N$ and any unit vector $v \in T_{x} N$. Then, $u$ is $ F $-unstable.
	\end{thm}
	\begin{proof} Let $ \{e_i\}_{i=1}^{m} $ be a local orthonormal frame field
		on $ M $ and $ \omega $ be a parallel vector field in $ R^{n+p} $. Recall the index form of $ F $-harmonic map(c.f.  \cite{ara2001stability}),
		\begin{equation}\label{}
			\begin{split}
				I(V,V)=&\jifen{\FF\jiankuo{\tilde{\nabla}V,du}^2}\\
				&+\int_{M} F^{\prime}(\frac{|d u|^{2}}{2}) \cdot\left\{\langle\tilde{\nabla} V, \tilde{\nabla} V\rangle-\sum_{t=1}^{m} h( R^N\left(V, u_* e_{i}\right) u_* e_{i}, V)\right\} v_{g}\\
				&+ \jifen{\nabla^2 H(V,V)}.
			\end{split}
		\end{equation}
		We take $ V=\omega^\top $
		\begin{equation}\label{}
			\begin{split}
				I(\omega^\top,\omega^\top)=&\jifen{f\FF\jiankuo{\tilde{\nabla}\omega^\top,du}^2}\\
				&+\int_{M}f F^{\prime}(\frac{|d u|^{2}}{2}) \cdot\left\{\langle\tilde{\nabla} \omega^\top, \tilde{\nabla} \omega^\top\rangle+\sum_{t=1}^{m} h( R^N\left(\omega^\top, u_* e_{i}\right) \omega^\top,u_* e_{i} )\right\} v_{g}\\
				&+ \jifen{\nabla^2 H(V,V)},
			\end{split}
		\end{equation}
		So, by \cite[Theorem 10]{torbaghan2022stability}, we can get 
		\begin{equation}\label{}
			\begin{split}
				I(\omega^\top,\omega^\top)=&\jifen{F^{\prime\prime}(\frac{|d u|^{2}}{2}) \left \langle B(du(e_i),du(e_i)), \omega^\perp \right \rangle^2  }\\
				&+\int_{M} F^{\prime}(\frac{|d u|^{2}}{2}) \cdot\left\{|A^{\omega^\perp}(du(e_i)) |^2+\sum_{t=1}^{m} h( R^N\left(\omega^\top, u_* e_{i}\right) \omega^\top,u_* e_{i} )\right\} v_{g}\\
				&+ \jifen{\nabla^2 H(V,V)}.
			\end{split}
		\end{equation}
		Thus, we get 			
		\begin{equation*}
			\begin{split}
				\operatorname{Tr}_g I\leq &\jifen{\FF  |B(du(e_i),du(e_i))|^2  }\\
				&+\int_{M} F^{\prime}(\frac{|d u|^{2}}{2}) \cdot\left\{|\theta(v_i) -Ric(v_i,v_i)\right\}\d v_{g}\\
				\leq &\jifen{\FF |du |^4 h(x)  }\\
				&+\int_{M} F^{\prime}(\frac{|d u|^{2}}{2})|du |^2 \cdot\left\{|\theta(v_i) -Ric(v_i,v_i)\right\} \d v_{g}.
			\end{split}
		\end{equation*}
		By our assumption, we get 
		\begin{equation*}
			\begin{split}
				2\frac{F^{\prime\prime}(\frac{|d u|^{2}}{2})}{F^{\prime}(\frac{|d u|^{2}}{2})}|du|^2 h+\theta(v)<\operatorname{Ric}(v, v),
			\end{split}
		\end{equation*}
		Thus, 
		\begin{equation*}
			\begin{split}
				\operatorname{Tr}_gI\leq 0.
			\end{split}
		\end{equation*}
	\end{proof}	
	\section{$F$ symphonic map}\label{sec4}
	\subsection{ $F$ symphonic map with potential from $ \Phi
		$-SSU manifold  }
	\begin{lem}[\cite{ara2001instability}]\label{dkl}
		For any constant $a>0$, there is a strictly increasing and convex $C^{2}$ function $F:[0, \infty) \rightarrow[0, \infty)$ such that $t \cdot F^{\prime \prime}(t)<a \cdot F^{\prime}(t)$ for any $t>0$.

	\end{lem}
	\begin{rem}By \cite[Lemma 4.9]{ara2001instability}, the follow function $ F $ satisfies the conculsion.
		
		(1) $F_{1}(t)=t^{b+1}, 0<b<a$,
		
		(2) $F_{2, n}(t)=\sum_{i=1}^{n} a_{i} t^{i}, n<a+1, a_{1}>0, a_{i} \geq 0(i=2, \cdots, n)$,
		
		(3) $F_{3}(t)=\int_{0}^{t} e^{\int_{0}^{s} G(u) d u} d s$, where $G(u)$ is a continuous function and $u \cdot G(u)<a$
	\end{rem}

	\begin{lem}[c.f. see \cite{Li2017NonexistenceOS} and\cite{han2014monotonicity}]
		Let $ u_{s,t}:(M^m,g)\to (N^n,h) $ be  a smooth deformation of u such that such that $ u_0=u, V=\frac{\partial u_t}{\partial t}|_{t=0} ,W=\frac{\partial u_t}{\partial s}|_{t=0}$, then 
		\begin{equation}\label{second}
			\begin{split}
				\left.\frac{\partial^{2}}{\partial s \partial t} \Phi_{sym}\left(u_{s, t}\right)\right|_{s, t=0}=&\int_{M} \operatorname{HessH}(V, W) d v_{g} \\
				&+\int_{M} F^{\prime \prime}\left(\frac{\left\|u^{*} h\right\|^{2}}{4}\right)\left\langle \widetilde{\nabla} V, \sigma_{u}\right\rangle \left\langle \widetilde{\nabla} W, \sigma_{u}\right\rangle d v_{g} \\
				&	+\int_{M} F^{\prime}\left(\frac{\left\|u^{*} h\right\|^{2}}{4}\right) \sum_{i, j=1}^{m} h\left(\widetilde{\nabla}_{e_{i}} V, \widetilde{\nabla}_{e_{j}} W\right) h\left(d u\left(e_{i}\right), d u\left(e_{j}\right)\right) d v_{g} \\
				&	+\int_{M} F^{\prime}\left(\frac{\left\|u^{*} h\right\|^{2}}{4}\right) \sum_{i, j=1}^{m} h\left(\widetilde{\nabla}_{e_{i}} V, d u\left(e_{j}\right)\right) h\left(\widetilde{\nabla}_{e_{i}} W, d u\left(e_{j}\right)\right) d v_{g} \\
				&	+\int_{M} F^{\prime}\left(\frac{\left\|u^{*} h\right\|^{2}}{4}\right) \sum_{i, j=1}^{m} h\left(\widetilde{\nabla}_{e_{i}} V, d u\left(e_{j}\right)\right) h\left(d u\left(e_{i}\right), \widetilde{\nabla}_{e_{j}} W\right) d v_{g} \\
				&	+\int_{M} F^{\prime}\left(\frac{\left\|u^{*} h\right\|^{2}}{4}\right) \sum_{i, j=1}^{m} h\left(R^{N}\left(V, d u\left(e_{i}\right)\right) W, d u\left(e_{j}\right)\right) h\left(d u\left(e_{i}\right), d u\left(e_{j}\right)\right) d v_{g},
			\end{split}
		\end{equation}
		
		where $ \sigma_{u}(\cdot)=h(du(\cdot),du(e_j))du(e_j). $
	\end{lem}

		
		\begin{defn}
			Let RHS of \eqref{second} be the index form $ I(V,W) $. $ F $-symphonic map $ u  $  with potential is called $ \Phi_{sym}$-stable  if $ I(V,V)\geq 0 $ for any nonzero vector filed $V$. Otherwise, it is called $ \Phi_{sym}$-unstable. If the identiy map is stable, we call $ M $   $ \Phi_{sym}  $-stable, otherwise, it is called  $ \Phi_{sym}  $-unstable.
		\end{defn}
	
	\begin{thm}\label{thm2.1}
		Let $ (M^m, g) $ be a compact $ \Phi
		$-$ SSU  $ manifold. For  $x \in M$
		$$ a=\min _{X \in U M, Y \in U M} \frac{-\left\langle Q_{x}^{M}(X), X\right\rangle_{M}}{8|B(X, X)|_{\mathbf{R}^{r}}|B(Y, Y)|_{\mathbf{R}^{r}}}>0,$$  Let    $ F $ be the positive  function determined by Lemma \ref{dkl}, $ \nabla^2 H \leq 0 $.
		Let $ u $ be  stable $ \Phi_{sym} $-
		harmonic map with potential  from $ (M^m, g)  $ into any Riemannian manifold $ N $, then $ u $ is constant.
	\end{thm}

		\begin{proof}
			
		The proof is almost the same as  in \cite{2021The} after replacing $ S_\Phi(u) $ by $ u^*h $ and replacing  $ \sigma_\Phi $ by $ \sigma_{u } $. We use the same notations as in the proof of Theorem 6.1 in \cite{2021The}. We choose an orthogonal frame field $\left\{e_{1}, \cdots, e_{m+p}\right\}$ of $R^{m+p}$ such that $\left\{e_{i}\right\}_{i=1}^{m}$ are tangent to $M^{m},\left\{e_{\alpha}\right\}_{\alpha=m+1}^{m+p}$ are normal to $M^{m}$ and $\left.\nabla_{e_{i}} e_{j}\right|_{x}=0$, where $x$ is a fixed point of $M$. We take a fixed orthonormal basis of $R^{m+p}$ denoted by $E_{D}, D=$ $1, \cdots, m+p$ and set
		\[
		V_{D}=\sum_{i=1}^{m} v_{D}^{i} e_{i},  v_{D}^{i}=\left\langle E_{D}, e_{i}\right\rangle,  v_{D}^{\alpha}=\left\langle E_{D}, e_{\alpha}\right\rangle,  \alpha=m+1, \cdots, m+p,
		\]
		where $\langle\cdot, \cdot\rangle$ is the canonical Euclidean inner product. Here, we cite the three formulas (37)(38)(39) in \cite{2021The} as follows:
		\begin{equation*}
			\begin{split}
					\sum_{D=1}^{m+p} v_{D}^{i} v_{D}^{j} &=\sum_{D=1}^{m+p}\left\langle E_{D}, e_{i}\right\rangle\left\langle E_{D}, e_{j}\right\rangle=\delta_{i j}, \quad i, j=1, \cdots, m, \\
				\nabla_{e_{i}} V_{D} &=\sum_{\alpha=m+1}^{m+p} \sum_{j=1}^{m} B_{i j}^{\alpha} v_{D}^{\alpha} e_{j}, \\
				\tilde{\nabla}_{e_{i}} d u\left(V_{D}\right) &=\sum_{\alpha=m+1}^{m+p} \sum_{k=1}^{m} v_{D}^{\alpha} B_{i k}^{\alpha} d u\left(e_{k}\right)+\sum_{k=1}^{m} v_{D}^{k} \widetilde{\nabla}_{e_{i}} d u\left(e_{k}\right),
			\end{split}
		\end{equation*}

	Using the witenzbock formula, 
	$$-R^{N}\left(u_{*} V_{D}, \mathrm{d} u\left(e_{i}\right)\right) \mathrm{d} u\left(e_{i}\right)+u_{*} \operatorname{Ric}^{M^{n}}\left(V_{D}\right)=\Delta \mathrm{d} u\left(V_{D}\right)+\tilde{\nabla}^{2} \mathrm{d} u\left(V_{D}\right).$$	
	and	
	\begin{equation*}
		\begin{split}
			&\sum_{D} \int_{M} \text{Hess} H(du(V_D),du(V_D))		+\sum_{D} \int_{M}F^{\prime}(\frac{|u^{*}h |^2}{4})\left\langle(\Delta d u)\left(V_{D}\right), \sigma_{u}\left(V_{D}\right)\right\rangle d v_{g}\\
			=&\int_{M} \sum_{i, j, D} v_{D}^{i} v_{D}^{j}\left\langle(\Delta d u)\left(e_{i}\right), \sigma_{u}\left(e_{j}\right)\right\rangle d v_{g} \\
			=&\int_{M} \sum_{i}\left\langle(\Delta d u)\left(e_{i}\right), \sigma_{u}\left(e_{i}\right)\right\rangle d v_{g}=\int_{M}\left\langle(\Delta d u), \sigma_{u}\right\rangle d v_{g}=\int_{M}\left\langle\delta d u, \delta \sigma_{u}\right\rangle d v_{g} \\
			=&-\int_{M}\left\langle\delta d u, \operatorname{div}\left(\sigma_{u}\right)\right\rangle d v_{g}=0.
		\end{split}
	\end{equation*}
	
By \eqref{second}, 	we have

	\begin{equation*}
		\begin{split}
			&\sum_{D=1}^{m+p}I(du(V_D),du(V_D))\\
			=&	\int_{M} F^{\prime \prime}\left(\frac{\left\|u^{*} h\right\|^{2}}{4}\right)\left\langle \widetilde{\nabla} V_D, \sigma_{u}\right\rangle \left\langle \widetilde{\nabla} V_D, \sigma_{u}\right\rangle d v_{g} \\
			&	+\int_{M} F^{\prime}\left(\frac{\left\|u^{*} h\right\|^{2}}{4}\right) \sum_{i, j=1}^{m} h\left(\widetilde{\nabla}_{e_{i}} V_D, \widetilde{\nabla}_{e_{j}} V_D\right) h\left(d u\left(e_{i}\right), d u\left(e_{j}\right)\right) d v_{g} \\
			&	+\int_{M} F^{\prime}\left(\frac{\left\|u^{*} h\right\|^{2}}{4}\right) \sum_{i, j=1}^{m} h\left(\widetilde{\nabla}_{e_{i}} V_D, d u\left(e_{j}\right)\right) h\left(\widetilde{\nabla}_{e_{i}} V_D, d u\left(e_{j}\right)\right) d v_{g} \\
			&	+\int_{M} F^{\prime}\left(\frac{\left\|u^{*} h\right\|^{2}}{4}\right) \sum_{i, j=1}^{m} h\left(\widetilde{\nabla}_{e_{i}} V_D, d u\left(e_{j}\right)\right) h\left(d u\left(e_{i}\right), \widetilde{\nabla}_{e_{j}} V_D\right) d v_{g} \\
			&	+\int_{M} F^{\prime}\left(\frac{\left\|u^{*} h\right\|^{2}}{4}\right) \sum_{i, j=1}^{m} h\left(du(Ric^M(e_i)), d u\left(e_{j}\right)\right) h\left(d u\left(e_{i}\right), d u\left(e_{j}\right)\right) d v_{g}	\\
			&	+\int_{M} F^{\prime}\left(\frac{\left\|u^{*} h\right\|^{2}}{4}\right) \sum_{i, j=1}^{m} h\left((\nabla^2 du)(e_i)), d u\left(e_{j}\right)\right) h\left(d u\left(e_{i}\right), d u\left(e_{j}\right)\right) d v_{g}\\
			&:=J_1+J_2+J_3+	J_4+J_5+J_6+J_7.
		\end{split}
	\end{equation*}

	\begin{equation}\label{k9}
		\begin{split}
			J_1=&F^{\prime \prime}\left(\frac{\left\|u^{*}h\right\|^{2}}{4}\right) \sum_{D}\left\langle\tilde{\nabla} d u\left(V_{D}\right), \sigma_{u}\right\rangle^{2} \\
			=& \sum_{A} F^{\prime \prime}\left(\frac{\left\|u^{*}h\right\|^{2}}{4}\right)\left[\sum_{i}\left\langle\tilde{\nabla}_{e_{i}} d u\left(V_{D}\right), \sigma_{u}\left(e_{i}\right)\right\rangle\right]\left[\sum_{j}\left\langle\tilde{\nabla}_{e_{j}} d u\left(V_{D}\right), \sigma_{u}\left(e_{j}\right)\right\rangle\right] \\
			=& F^{\prime \prime}\left(\frac{\left\|u^{*}h\right\|^{2}}{4}\right) \sum_{A} \left[\sum_{i}\left(-v_{D}^{\alpha}B_{ik}^{\alpha} h\left(d u\left(e_{k}\right), \sigma_{u}\left(e_{i}\right)\right)+v_{D}^{k} h\left(\tilde{\nabla}_{e_{i}} d u\left(e_{k}\right), \sigma_{u}\left(e_{i}\right)\right)\right)\right] \\
			&\times \sum_{A} \left[\sum_{i}\left(-v_{D}^{\alpha}B_{jl}^{\alpha} h\left(d u\left(e_{l}\right), \sigma_{u}\left(e_{j}\right)\right)+v_{D}^{k} h\left(\tilde{\nabla}_{e_{j}} d u\left(e_{k}\right), \sigma_{u}\left(e_{j}\right)\right)\right)\right] \\
			=& F^{\prime \prime}\left(\frac{\left\|u^{*}h\right\|^{2}}{4}\right) \sum_{D} B_{ik}^{\alpha} B_{jl}^{\alpha}\left[\sum_{i} h\left(d u\left(e_{i}\right), \sigma_{u}\left(e_{k}\right)\right)\right] \left[\sum_{i} h\left(d u\left(e_{j}\right), \sigma_{u}\left(e_{l}\right)\right)\right] \\
			&+F^{\prime \prime}\left(\frac{\left\|u^{*}h\right\|^{2}}{4}\right) \sum_{D}\left[\sum_{l} v_{A}^{l}\left[\sum_{i} h\left(\left(\nabla_{e_{i}} d u\right)\left(e_{l}\right), \sigma_{u}\left(e_{i}\right)\right)\right]\right]^{2} \\
			=& F^{\prime \prime}\left(\frac{\left\|u^{*}h\right\|^{2}}{4}\right) \sum_{D} B_{ik}^{\alpha} B_{jl}^{\alpha}\left[\sum_{i} h\left(d u\left(e_{i}\right), \sigma_{u}\left(e_{k}\right)\right)\right] \left[\sum_{i} h\left(d u\left(e_{j}\right), \sigma_{u}\left(e_{l}\right)\right)\right] \\
			&+ F^{\prime \prime}\left(\frac{\left\|u^{*}h\right\|^{2}}{4}\right)\left[\sum_{l}\left[\sum_{i} h\left(\left(\nabla_{e_{l}} d u\right)\left(e_{i}\right), \sigma_{u}\left(e_{i}\right)\right)\right]^{2}\right],
		\end{split}
	\end{equation}
	
By (43)-(48) in \cite{2021The}, we have 
	\begin{equation}\label{k8}
		\begin{split}
			J_2&=F^{\prime}\left(\frac{\left\|u^{*}h\right\|^{2}}{4}\right) \sum_{i, j,D} h\left(\widetilde{\nabla}_{e_i} du(V_D), \widetilde{\nabla}_{e_j} du(V_D)\right)h\left(\mathrm{~d} u\left(e_{i}\right), \mathrm{d} u\left(e_{j}\right)\right)  \\
			&=F^{\prime}\left(\frac{\left\|u^{*}h\right\|^{2}}{4}\right) 
			\bigg\{\sum_{i, j, \alpha} h\left(d u\left(A^{\alpha}\left(e_{i}\right)\right), d u\left(A^{\alpha}\left(e_{j}\right)\right)\right)h\left(d u\left(e_{i}\right), d u\left(e_{j}\right)\right) \\
			&+\sum_{i, j, k} h\left(\left(\nabla_{e_{k}} d u\right)\left(e_{i}\right),\left(\nabla_{e_{k}} d u\right)\left(e_{j}\right)\right) h\left(d u\left(e_{i}\right), d u\left(e_{j}\right)\right)\bigg\}, 
		\end{split}
	\end{equation}	
	and
	\begin{equation}\label{k10}
		\begin{split}
			J_3=&F^{\prime}\left(\frac{\left\|u^{*}h\right\|^{2}}{4}\right) \sum_{i, j=1}^{m}\sum_D h\left(\widetilde{\nabla}_{e_i} du(V_D), \mathrm{~d} u\left(e_{j}\right)\right) h\left(\widetilde{\nabla}_{e_i} du(V_D), \mathrm{~d} u\left(e_{j}\right)\right)\\
			=&F^{\prime}\left(\frac{\left\|u^{*}h\right\|^{2}}{4}\right)\sum h(du(A^\alpha(e_i)),du(e_j))h(du(e_i),du(A^\alpha(e_j)))\\
			&+F^{\prime}\left(\frac{\left\|u^{*}h\right\|^{2}}{4}\right)\sum h((\nabla_{e_k}du)((e_i)),du(e_j))h(du(e_i),(\nabla_{e_k}du)((e_j))),
		\end{split}
	\end{equation}
	and
	\begin{equation}\label{k11}
		\begin{split}
			J_4=&F^{\prime}\left(\frac{\left\|u^{*}h\right\|^{2}}{4}\right) \sum_{i, j=1}^{m}\sum_D h\left(\widetilde{\nabla}_{e_i} du(V_D), \mathrm{~d} u\left(e_{j}\right)\right) h\left(\widetilde{\nabla}_{e_j} du(V_D), \mathrm{~d} u\left(e_{i}\right)\right)\\
			=&F^{\prime}\left(\frac{\left\|u^{*}h\right\|^{2}}{4}\right)\sum h(du(A^\alpha A^\alpha(e_i)),du(e_j))h(du(e_i),du((e_j)))\\
			&+F^{\prime}\left(\frac{\left\|u^{*}h\right\|^{2}}{4}\right)\sum h((\nabla_{e_k}du)((e_i)),du(e_j))h((\nabla_{e_k}du)((e_i)),du(e_j)),
		\end{split}
	\end{equation}
and

	\begin{equation}\label{k13}
		\begin{split}
			J_5=&\frac{m-4}{2} F^{\prime}\left(\frac{\left\|u^{*}h\right\|{ }^{2}}{4}\right) \sum_{i=1}^{m} h\left(\widetilde{\nabla}_{e_i} du(V_D), \mathrm{~d} u\left(e_{i}\right)\right)  	\sum_{j=1}^{m} h\left(\tilde{\nabla}_{e_j} du(V_D), \mathrm{~d} u\left(e_{j}\right)\right) 
			\\
			=&\frac{m-4}{2} F^{\prime}\left(\frac{\left\|u^{*}h\right\|{ }^{2}}{4}\right)\sum h(du(A^\alpha(e_i)),du(e_i))h(du(e_j),du(A^\alpha(e_j)))\\
			&+\frac{m-4}{2} F^{\prime}\left(\frac{\left\|u^{*}h\right\|{ }^{2}}{4}\right)\sum h((\nabla_{e_k}du)((e_i)),du(e_i))h((\nabla_{e_k}du)((e_j)),du(e_j)),
		\end{split}
	\end{equation}
and

	\begin{equation}\label{k11}
		\begin{split}
			J_6=&F^{\prime}\left(\frac{\left\|u^{*}h\right\|^{2}}{4}\right) \sum_{i, j=1}^{m} h\left(du(Ric^M(e_i)), \mathrm{~d} u\left(e_{j}\right)\right)  	 h\left(\mathrm{~d} u\left(e_{i}\right), \mathrm{d} u\left(e_{j}\right)\right)\\
			=& F^{\prime}\left(\frac{\left\|u^{*}h\right\|^{2}}{4}\right) \sum_{i, j=1}^{m} h\left(du(\mathrm{trace}(A^\alpha)A^\alpha(e_i)), \mathrm{~d} u\left(e_{j}\right)\right)   h\left(\mathrm{~d} u\left(e_{i}\right), \mathrm{d} u\left(e_{j}\right)\right)\\
			&-F^{\prime}\left(\frac{\left\|u^{*}h\right\|^{2}}{4}\right) \sum_{i, j=1}^{m} h\left(du(A^\alpha A^\alpha(e_i)), \mathrm{~d} u\left(e_{j}\right)\right)  	 h\left(\mathrm{~d} u\left(e_{i}\right), \mathrm{d} u\left(e_{j}\right)\right).
		\end{split}
	\end{equation}

	The last term $ J_7 $ can help us to cancel several terms including $ \nabla du $ in $ J_1$-$J_6 $,	
	\begin{equation}\label{k11}
		\begin{split}
			J_7=& F^{\prime}\left(\frac{\left\|S_{u}\right\|{ }^{2}}{4}\right)\sum_{i, j} h\left(\left(\nabla^{2} d u\right)\left(e_{i}\right), d u\left(e_{j}\right)\right)h\left(d u\left(e_{i}\right), d u\left(e_{j}\right)\right)\\
			=&\sum_{i, j, k} e_{k}\left( F^{\prime}\left(\frac{\left\|S_{u}\right\|{ }^{2}}{4}\right)h\left(\nabla_{e_{k}} d u\left(e_{i}\right),  du\left(e_{j}\right)\right)h\left(d u\left(e_{i}\right), d u\left(e_{j}\right)\right)\right)\\
			&- F^{\prime}\left(\frac{\left\|S_{u}\right\|{ }^{2}}{4}\right)\sum_{i, j, k} h\left(\left(\nabla_{e_{k}} d u\right)\left(e_{i}\right),\left(\nabla_{e_{k}} d u\right)\left(e_{j}\right)\right)h\left(d u\left(e_{i}\right), d u\left(e_{j}\right)\right)\\
			&- F^{\prime}\left(\frac{\left\|S_{u}\right\|{ }^{2}}{4}\right)\sum_{i, j, k} h\left(\left(\nabla_{e_{k}} d u\right)\left(e_{i}\right), d u\left(e_{j}\right)\right) h\left(d u\left(e_{i}\right),\left(\nabla_{e_{k}} d u\right)\left(e_{j}\right)\right)\\
			&- F^{\prime}\left(\frac{\left\|S_{u}\right\|{ }^{2}}{4}\right)\sum_{i, j, k} h\left(\left(\nabla_{e_{k}} d u\right)\left(e_{i}\right), d u\left(e_{j}\right)\right) h\left(\left(\nabla_{e_{k}} d u\right)\left(e_{i}\right), d u\left(e_{j}\right)\right)	\\	
			&- F^{\prime}\left(\frac{\left\|S_{u}\right\|{ }^{2}}{4}\right)\frac{m-4}{2} h\left(\left(\nabla_{e_{k}} d u\right)\left(e_{i}\right), d u\left(e_{i}\right)\right) h\left(\left(\nabla_{e_{k}} d u\right)\left(e_{j}\right), d u\left(e_{j}\right)\right)\\
			&-F^{\prime \prime}\left(\frac{\left\|u^{*}h\right\|^{2}}{4}\right)\sum_{l}\left[\sum_{i} h\left(\left(\nabla_{e_{l}} d u\right)\left(e_{i}\right), \sigma_{u}\left(e_{i}\right)\right)\right]^{2},
		\end{split}
	\end{equation}
	where we have used the formula 
	\begin{equation}\label{k14}
		\begin{split}
			e_kF^{\prime}\left(\frac{\left\|u^{*}h\right\|^{2}}{4}\right)=F^{\prime \prime}\left(\frac{\left\|u^{*}h\right\|^{2}}{4}\right) h\left(\left(\nabla_{e_{k}} d u\right)\left(e_{i}\right), \sigma_{u}\left(e_{i}\right)\right).
		\end{split}
	\end{equation}	
Combing \eqref{k9}	\eqref{k10}\eqref{k11}\eqref{k13}\eqref{k14}\eqref{k15}, we get 	
	\begin{equation}\label{k15}
		\begin{split}
			&\sum_{D} I\left(d u\left(V_{D}\right), d u\left(V_{D}\right)\right)\\
			=&F^{\prime \prime}\left(\frac{\left\|u^{*}h\right\|^{2}}{4}\right) \sum_{A} B_{ik}^{\alpha} B_{jl}^{\alpha}\left[\sum_{i} h\left(d u\left(e_{i}\right), \sigma_{u}\left(e_{k}\right)\right)\right] \left[\sum_{i} h\left(d u\left(e_{j}\right), \sigma_{u}\left(e_{l}\right)\right)\right] \\
			&+\int_{M}F^{\prime}\left(\frac{\left\|u^{*}h\right\|^{2}}{4}\right) \sum_{i, j, \alpha} h\left(d u\left(A^{\alpha}\left(e_{i}\right)\right), d u\left(A^{\alpha}\left(e_{j}\right)\right)\right) h\left(d u\left(e_{i}\right), d u\left(e_{j}\right)\right)d v_{g} \\
			&+\int_{M}F^{\prime}\left(\frac{\left\|u^{*}h\right\|^{2}}{4}\right) \sum_{i, j, \alpha} h\left(d u\left(A^{\alpha}\left(e_{i}\right)\right), d u\left(e_{j}\right)\right) h\left(d u\left(e_{i}\right), d u\left(A^{\alpha}\left(e_{j}\right)\right)\right) d v_{g} \\
			&+\int_{M}F^{\prime}\left(\frac{\left\|u^{*}h\right\|^{2}}{4}\right) \sum_{i, j, \alpha} h\left(d u\left(A^{\alpha} A^{\alpha}\left(e_{i}\right), d u\left(e_{j}\right)\right) h\left(d u\left(e_{i}\right), d u\left(e_{j}\right)\right) d v_{g}\right. \\
			&+\frac{m-4}{2} \int_{M} F^{\prime}\left(\frac{\left\|u^{*}h\right\|^{2}}{4}\right)\sum_{i, j, \alpha} h\left(d u\left(A^{\alpha}\left(e_{i}\right)\right), d u\left(e_{i}\right)\right) h\left(d u\left(A^{\alpha}\left(e_{j}\right)\right), d u\left(e_{j}\right)\right) d v_{g} \\
			-&\int_{M}F^{\prime}\left(\frac{\left\|u^{*}h\right\|^{2}}{4}\right) \sum_{i, j, \alpha} h\left(d u\left(\operatorname{trace}\left(A^{\alpha}\right) A^{\alpha}\left(e_{i}\right)\right), d u\left(e_{j}\right)\right) h\left(d u\left(e_{i}\right), d u\left(e_{j}\right)\right) d v_{g} \\
		&+\int_{M}F^{\prime}\left(\frac{\left\|u^{*}h\right\|^{2}}{4}\right) \sum_{i, j, \alpha} h\left(d u\left(A^{\alpha} A^{\alpha}\left(e_{i}\right)\right), d u\left(e_{j}\right)\right)h\left(d u\left(e_{i}\right), d u\left(e_{j}\right)\right)d v_{g}.
		\end{split}
	\end{equation}

	Choosing orthornormal frame $\{e_1,e_2,\cdots,e_n\}  $ such that $ h(du(e_i),du(e_j))=\lambda_i^2 \delta_{ij} $, we get 
	\begin{equation}\label{870}
		\begin{split}
		F^{\prime \prime}\left(\frac{\left\|u^{*}h\right\|^{2}}{4}\right) \sum B_{ik}^{\alpha} B_{jl}^{\alpha}\left[\sum_{i} h\left(d u\left(e_{i}\right), \sigma_{u}\left(e_{k}\right)\right)\right] \left[\sum_{i} h\left(d u\left(e_{j}\right), \sigma_{u}\left(e_{l}\right)\right)\right] \\
			=F^{\prime \prime}\left(\frac{\left\|u^{*}h\right\|^{2}}{4}\right) \sum\left \langle B(e_i,e_i) ,B(e_j,e_j) \right \rangle \lambda_i^4\lambda_j^4.
		\end{split}
	\end{equation}

	So, by \eqref{k15}\eqref{870} and (50)-(58)  in \cite{2021The}, we get 	
	\begin{equation*}
		\begin{split}
			&\sum_{D} I\left(d u\left(V_{D}\right), d u\left(V_{D}\right)\right) \\
			\leq&  \int_{M}  F^{\prime\prime}\left(\frac{\left\|u^{*}h\right\|^{2}}{4}\right)\left \langle B(e_i,e_i) ,B(e_j,e_j) \right \rangle \lambda_i^4\lambda_j^4 \d v_g \\
			&+\int_{M}  F^{\prime}\left(\frac{\left\|u^{*}h\right\|^{2}}{4}\right)\sum_{i} \lambda_{i}^{4}\sum_{j}\left(4\left\langle B\left(e_{i}, e_{j}\right), B\left(e_{i}, e_{j}\right)\right\rangle-\left\langle B\left(e_{i}, e_{i}\right), B\left(e_{j}, e_{j}\right)\right\rangle\right)\d v_g\\
				\leq&\frac{1}{2}  \int_{M}  F^{\prime}\left(\frac{\left\|u^{*}h\right\|^{2}}{4}\right)\sum_{i} \lambda_{i}^{4} \sum_{j}\big(4\left\langle B\left(e_{i}, e_{j}\right), B\left(e_{i}, e_{j}\right)\right\rangle-\left\langle B\left(e_{i}, e_{i}\right), B\left(e_{j}, e_{j}\right)\right\rangle\big)\d v_g
		\end{split}
	\end{equation*}
	
\end{proof}
From the above proof ,we can see that 
		\begin{thm}\label{thmc}
		Let $ (M^m, g) $ be a compact $ \Phi
		$-SSU manifold and $ F^{\prime\prime} \leq 0 $, $\nabla^2 H$ is semipositive. Then every  $ \Phi_{sym} $-stable $ F $-symphonic map $u$ from $(M^m, g)$ with potential into any Riemannian manifold $ N $ is constant.
	\end{thm}
	\subsection{ $F$-symphonic map with potential  into $ \Phi
		$-SSU manifold  }
	
	\begin{thm}\label{thm2.1}
		Let $ (M^m, g) $ be a compact $ \Phi
		$-$ SSU  $ manifold. For  $x \in M$
		$$ a=\min _{X \in U M, Y \in U M} \frac{-\left\langle Q_{x}^{M}(X), X\right\rangle_{M}}{8|B(X, X)|_{\mathbf{R}^{r}}|B(Y, Y)|_{\mathbf{R}^{r}}}>0,$$  Let    $ F $ be the positive  function determined by Lemma \ref{dkl}, $ \nabla^2 H \leq 0 $.
		Let $ u $ be  stable $ \Phi_{sym} $-
		harmonic map with potential  from $ (N, g)  $ into any Riemannian manifold $(M,g)$, then $ u $ is constant.
	\end{thm}

\begin{proof}
	We use the same notations  as in the proof of Theorem 7.1 in \cite{2021The}.
 Let   $\left\{e_{1}, \cdots, e_{m}\right\}$ be a local orthonormal frame field of $M$ . Let   $\left\{\epsilon_{1}, \cdots, \epsilon_{n}, \epsilon_{n+1}, \cdots, \epsilon_{n+p}\right\}$ be an orthonormal frame field of $R^{n+p}$, such that $\left\{\epsilon_{i}, \cdots, \epsilon_{n}\right\}$ are tangent to $N^{n}, \epsilon_{n+1}, \cdots, \epsilon_{n+p}$ are normal to $N^{n}$ and $\left.{}^{N}\nabla_{\epsilon_{b}} \epsilon_{c}\right|_{u(x)}=0$, where $x$ is a fixed point of $M$. As  in \cite{2021The}, we fix an orthonormal basis $E_{D}$ of $R^{m+p}$, for $D=1, \cdots, m+p$ and set
\begin{equation*}
	\begin{split}
		V_{D}&=\sum_{b=1}^{n} v_{D}^{b} \epsilon_{a},  v_{D}^{b}=\left\langle E_{D}, \epsilon_{b}\right\rangle,\\
		v_{D}^{\alpha}&=\left\langle E_{D}, \epsilon_{\alpha}\right\rangle, \quad   for \quad  \alpha=n+1, \cdots, n+p,\\
		\nabla_{\epsilon_{b}} V_{D}&=\sum_{\alpha=n+1}^{n+p} \sum_{c=1}^{n} v_{D}^{\alpha} B_{b c}^{\alpha} \epsilon_{c}, 1 \leq b \leq n;
	\end{split}
\end{equation*}

choose local frame such that 
\begin{equation}\label{kkk}
	\begin{split}
		\sum_{i=1}^{m}u_{i}^{b}u_{i}^{c}=\lambda_b^2 \delta_{bc}.
	\end{split}
\end{equation}

Let $ I(V_D,V_D) $ be the index form in  \eqref{second}. 
\begin{equation*}
	\begin{split}
		I(V_D,V_D)=\sum_{i=0}^{7}J_i
	\end{split}
\end{equation*}
	Inspired by the formula (62)-(66) in \cite{2021The}, similar to \eqref{k9}, using (59) in \cite{2021The},   we have 
	
	By \eqref{kkk}, a direct computation gives  
	\begin{equation*}
		\begin{split}
			&F^{\prime \prime}\left(\frac{\left\|u^{*}h\right\|^{2}}{4}\right) \sum_{A}\left\langle\tilde{\nabla} \left(V_{D}\right), \sigma_{u}\right\rangle^{2} \\
			=&F^{\prime \prime}\left(\frac{\left\|u^{*}h\right\|^{2}}{4}\right)\sum \left \langle B(\epsilon_b,\epsilon_b), B(\epsilon_c, \epsilon_c) \right \rangle \lambda_b^4\lambda_c^4.
		\end{split}
	\end{equation*}
	
Using the formula (62)-(66) in \cite{2021The}, we can deal with the terms $ J_2,J_3,\cdots,J_7 $ as in the proof of Theorem \ref{thmc}. We have 
	
	\begin{equation*}
		\begin{split}
			&\sum_{D} I\left(V_{D}, V_{D}\right)\\
			\leq &\frac{1}{2}\int_{M}F^{\prime}\left(\frac{\left\|u^{*}h\right\|^{2}}{4}\right) \sum_{b} \lambda_{b}^{4} \sum_{c}\left(4\left\langle B\left(\epsilon_{b}, \epsilon_{c}\right), B\left(\epsilon_{b}, \epsilon_{c}\right)\right\rangle-\left\langle B\left(\epsilon_{b}, \epsilon_{b}\right), B\left(\epsilon_{c}, \epsilon_{c}\right)\right\rangle\right) d v_{g}.
		\end{split}
	\end{equation*}	
	Since $N$ is a $\Phi$-SSU manifold, if $u$ is not constant, we have
	\[
	\sum_{D} I\left(V_{D}, V_{D}\right)<0.
	\]
\end{proof}
From the above proof, we can get 
	\begin{thm}
	Let $ (M^m, g) $ be a compact $\Phi$-SSU manifold and $ F^{\prime\prime}\leq 0$. Then every $\Phi_{sym}$ stable $F$-symphonic map with potential $ H $  from $ N  $  into any Riemannian manifold $ M $ must be  constant.
\end{thm}

\begin{defn}
	A Riemannian manifold $ M $ is $\Phi$-strongly unstable ($\Phi$-
	SU) if it is neither the domain nor the target of any nonconstant smooth
	$\Phi_{sym}$-stable $F$-symphonic map,
\end{defn}

\begin{cor}
	Let $ (M^m, g) $ be a compact $ \Phi
$-$ SSU  $ manifold. For  $x \in M$,
$$ a=\min _{X \in U M, Y \in U M} \frac{-\left\langle Q_{x}^{M}(X), X\right\rangle_{M}}{8|B(X, X)|_{\mathbf{R}^{r}}|B(Y, Y)|_{\mathbf{R}^{r}}}>0,$$  Let    $ F $ be the positive  function determined by Lemma \ref{dkl}, $ \nabla^2 H \leq 0 $. Then $ M $ is $\Phi$-SU. 
\end{cor}

	\subsection{$ F $-symphonic map with  potential when  targeted manifold is $\delta$-pinched}

	In this section, we give Okayasu type theorem(cf. Howard \cite{Howard1985,Howard1986}).
	\begin{thm}
	Let $ u : (M^m,g) \to (N^n,h) $ be  $\Phi_{sym}$-stable conformal $F $-symphonic map with potential $ H $  with conformal factor $ \lambda $  from compact Riemannian manifold $ M $ into a compact simply connected $ \delta $-pinched n-dimensional
		Riemannian manifold $ N$. If one of the following conditions holds, 
		
		(1) there exists a constant  $ c_F<\frac{n-2-2m}{4}\lambda, n>2+2m$ and $\nabla^2 H \leq 0$ ,
		\begin{equation*}
		\begin{split}
			\left(\frac{4}{\lambda}c_{F}+2m+1\right)\left\{\frac{n}{4} k_{3}^{2}(\delta)+k_{3}(\delta)+1\right\}-\frac{2 \delta}{1+\delta}(n-1) <0.
			\end{split}
		\end{equation*}
			
		(2) $ F^{\prime\prime}(x)=F^{\prime}(x) $ and   $\lambda ^2 \leq \frac{n-2-2m}{m}, n>2+2m, $$\nabla^2 H \leq 0$
		\begin{equation*}
			\begin{split}
				\left(\lambda^2 m+2m+1\right)\left\{\frac{n}{4} k_{3}^{2}(\delta)+k_{3}(\delta)+1\right\}-\frac{2 \delta}{1+\delta}(n-1) <0.
			\end{split}
		\end{equation*}	
		Then $ u $ must be   a constant.

	\end{thm}
\begin{proof} 
	As in the proof of Theorem \ref{thm2.3}.
	As in \cite{MR2259738}, we assume the sectional curvature of N is equal to $ \frac{2\delta}{1+\delta} .$ Let  the vector  bundle $ E=TN\oplus \epsilon(N) ,$ here, $ \epsilon(N) $ is  the trivial bundle. As in\cite{ara2001stability} \cite{MR2259738}, we can define a   metric connection  $ \nabla^{\prime\prime} $ as follows:
	\begin{equation}\label{}
		\begin{split}
		&\nabla^{\prime\prime}_XY={}^{N}\nabla_XY-h(X,Y)e;\\
			&\nabla^{\prime\prime}_Xe=X,	
		\end{split}
	\end{equation}
We cite the functions in \cite{ara2001stability} or \cite{MR2259738}, 
\begin{equation*}
	\begin{split}
		k_{1}(\delta)=\frac{4(1-\delta)}{3 \delta}\left[1+\left(\sqrt{\delta} \sin \frac{1}{2} \pi \sqrt{\delta}\right)^{-1}\right],
		k_{2}(\delta)=\left[\frac{1}{2}(1+\delta)\right]^{-1}  k_{1}(\delta). \\
	\end{split}
\end{equation*}	
By \cite{ara2001stability} \cite{MR2259738}, we know that  there exists a flat connection $ \nabla^{\prime} $ such that 
\begin{equation*}
	\begin{split}
		\|\nabla^{\prime}-\nabla^{\prime\prime}\|\leq \frac{1}{2}k_3(\delta),
	\end{split}
\end{equation*}
where $ k_3(\delta) $ is defined as 
\begin{equation*}
	\begin{split}
				k_{3}(\delta)=k_{2}(\delta) \sqrt{1+\left(1-\frac{1}{24} \pi^{2}\left(k_{1}(\delta)\right)^{2}\right)^{-2}} .
	\end{split}
\end{equation*}

Along the line in \cite{ara2001stability}, take a cross section $ W  $ of $ E $ and let $ W^T $ denotes the $ TN $ component of $ W $ , by the second variation formula \eqref{second} for $ F $-sympohic map with potential, .
\begin{equation}\label{ckj}
	\begin{split}
		I(W^T,W^T)=&\int_{M} \operatorname{HessH}(W^T, W^T) d v_{g} \\
		&+\int_{M} F^{\prime \prime}\left(\frac{\left\|u^{*} h\right\|^{2}}{4}\right)\left\langle \widetilde{\nabla} W^T, \sigma_{u}\right\rangle \left\langle \widetilde{\nabla} W^T, \sigma_{u}\right\rangle d v_{g} \\
		&	+\int_{M} F^{\prime}\left(\frac{\left\|u^{*} h\right\|^{2}}{4}\right) \sum_{i, j=1}^{m} h\left(\widetilde{\nabla}_{e_{i}} W^T, \widetilde{\nabla}_{e_{j}} W^T\right) h\left(d u\left(e_{i}\right), d u\left(e_{j}\right)\right) d v_{g} \\
		&	+\int_{M} F^{\prime}\left(\frac{\left\|u^{*} h\right\|^{2}}{4}\right) \sum_{i, j=1}^{m} h\left(\widetilde{\nabla}_{e_{i}} W^T, d u\left(e_{j}\right)\right) h\left(\widetilde{\nabla}_{e_{i}} W^T, d u\left(e_{j}\right)\right) d v_{g} \\
		&	+\int_{M} F^{\prime}\left(\frac{\left\|u^{*} h\right\|^{2}}{4}\right) \sum_{i, j=1}^{m} h\left(\widetilde{\nabla}_{e_{i}} W^T, d u\left(e_{j}\right)\right) h\left(d u\left(e_{i}\right), \widetilde{\nabla}_{e_{j}} W^T\right) d v_{g} \\
		&	+\int_{M} F^{\prime}\left(\frac{\left\|u^{*} h\right\|^{2}}{4}\right) \sum_{i, j=1}^{m} h\left(R^{N}\left(W^T, d u\left(e_{i}\right)\right) W^T, d u\left(e_{j}\right)\right) h\left(d u\left(e_{i}\right), d u\left(e_{j}\right)\right) d v_{g}.
	\end{split}
\end{equation}

Let $\mathcal{W}:=\left\{W \in \Gamma(E) ; \nabla^{\prime} W=0\right\}$, then $\mathcal{W}$ with natural inner product is isomorphic to $\mathbf{R}^{n+1}$.

Taking an orthonormal basis $\left\{W_{1}, W_{2}, \ldots, W_{n}, W_{n+1}\right\}$ of $\mathcal{W}$ with respect to its natural inner product such that $ W_1,W_2, \cdots, W_n $ is tangent to $N$, we obtain

	Meanwhile, we observe that
	
	\begin{equation}\label{}
		\begin{split}
				\tilde{\nabla}_{e_i} W^T &={ }^N \nabla_{u_* e_i} W^T \\
			&=\nabla_{u_* e_i}^{\prime \prime} W^T+\left\langle W^T, u_* e_i\right\rangle e \\
			&=\nabla_{u_* e_i}^{\prime \prime}(W-\langle W, e\rangle e)+\left\langle W^T, u_* e_i\right\rangle e \\
			&=\left(\nabla_{u_* e_i}^{\prime \prime} W\right)^T-\langle W, e\rangle u_* e_i .
		\end{split}
	\end{equation}
	
	notice that 
	\begin{equation}\label{586}
		\begin{split}
			|\left(\nabla_{u_{*} e_i}^{\prime \prime} W\right)^T | =|\sum_{j=1}^{n}\left \langle \nabla_{u_{*} e_i}^{\prime \prime}\left( W\right) , W_j\right \rangle W_j|\leq \frac{nk_3(\delta)}{2}|u_*(e_i)| |W|.
		\end{split}
	\end{equation}
	
	Then, we have
	
\begin{equation}\label{}
	\begin{split}
		\sum_{i=1}^m\left|\tilde{\nabla}_{e_i} W^T\right|^2=& \sum_{i=1}^m\left|\left(\nabla_{u_*{e_i} }^{\prime \prime} W\right)^T\right|^2+\langle W, e\rangle^2|\mathrm{~d} u|^2 \\
		&-2 \sum_{i=1}^m\langle W, e\rangle\left\langle\nabla_{u_* e_i}^{\prime \prime} W^T, u_* e_i\right\rangle .
	\end{split}
\end{equation}
	
By  \eqref{586}
	\begin{equation*}
		\begin{split}
			\left \langle \tilde{\nabla}_{e_i} W^T,\sigma_{u}(e_i) \right \rangle =&	\left \langle \tilde{\nabla}_{e_i} W^T,du(e_j) \right \rangle  h(du(e_i),du(e_j))\\
			&\leq \lambda |\tilde{\nabla}_{e_i} W^T| |du| .
		\end{split}
	\end{equation*}
	So, \begin{equation*}
		\begin{split}
			\sum_{ij}	\left \langle \tilde{\nabla}_{e_i} W^T,\sigma_{u}(e_i) \right \rangle\left \langle \tilde{\nabla}_{e_j} W^T,\sigma_{u}(e_j) \right \rangle \leq  \lambda^2 |\tilde{\nabla}_{e_i} W^T|^2 |du|^2.
		\end{split}
	\end{equation*}
Since $ u $ is conformal map with conformal factor $ \lambda, $	
	\begin{equation*}
		\begin{split}
			\left \langle	\tilde{\nabla}_{e_i} W^T  ,\tilde{\nabla}_{e_j} W^T \right \rangle h(du(e_i),du(e_j)) \leq  \lambda |\tilde{\nabla}_{e_i} W^T|^2 .
		\end{split}
	\end{equation*}
By \eqref{586},
\begin{equation*}
	\begin{split}
		&\sum_{ij}\left \langle \tilde{\nabla}_{e_i} W^T,du(e_j) \right \rangle \left \langle \tilde{\nabla}_{e_i} W^T,du(e_j) \right \rangle \leq |\tilde{\nabla}_{e_i} W^T|^2 |du|^2
	\end{split}
\end{equation*}
Similarly, we also have 
\begin{equation*}
	\begin{split}
		&\sum_{ij}\left \langle \tilde{\nabla}_{e_i} W^T,du(e_j) \right \rangle \left \langle \tilde{\nabla}_{e_j} W^T,du(e_i) \right \rangle \leq |\tilde{\nabla}_{e_i} W^T|^2 |du|^2
	\end{split}
\end{equation*}
	Since $N$ is $\delta$-pinched, similar to \cite[(5.4)]{MR2259738}, we have 
	$$ h\left(R^N\left(W^T, u_* e_i\right) u_* e_i, W^T\right)h(du(e_i),du(e_j)) \geq \frac{2\lambda \delta}{1+\delta}\left\{\left|W^T\right|^2\left|u_* e_i\right|^2-\left\langle W^T, u_* e_i\right\rangle^2\right\}. $$
	Substituting ()() into (\ref{ckj}), we obtain
	\begin{equation*}
		\begin{split}
			I\left(W^T, W^T\right) \leq& \int_M F^{\prime}\left(\frac{|u^* h|^2}{2}\right) \lambda q(W) \mathrm{d}v_g,
		\end{split}
	\end{equation*}
	
where 
\begin{equation*}
	\begin{split}
		q(W)=&\left(\frac{4}{\lambda}c_{F}+2m+1\right)\left\{\sum_{i=1}^{m}\left|\left(\nabla_{u_{e} e_{i}}^{\prime \prime} W\right)^{T}\right|^{2}+\langle W, e\rangle^{2}|\mathrm{~d} u|^{2}\right.\\
			&\left.\quad-2 \sum_{i=1}^{m}\langle W, e\rangle\left\langle\nabla_{u_{*} e_{i}}^{\prime \prime} W^{T}, u_{*} e_{i}\right\rangle\right\} \\
			&-\frac{2 \delta}{1+\delta} \sum_{i=1}^{m}\left\{\left|W^{T}\right|^{2}\left|u_{*} e_{i}\right|^{2}-\left\langle W^{T}, u_{*} e_{i}\right\rangle^{2}\right\} .
	\end{split}
\end{equation*}
For the  quadratic form $q$ on $\mathcal{W}$, by the argument in  \cite{MR2259738}, we have 
\begin{equation*}
	\begin{split}
		\operatorname{trace}(q)\leq \Phi_{n,F}(\delta)|du |^2.
	\end{split}
\end{equation*}
where 
\begin{equation*}
	\begin{split}
		\Phi_{n,F}(\delta)=\left(\frac{4}{m\lambda} c_{F}+1\right)\left\{\frac{n}{4} k_{3}^{2}(\delta)+k_{3}(\delta)+1\right\}-\frac{2 \delta}{1+\delta}(n-1).
	\end{split}
\end{equation*}
Taking traces , we get 	
	\begin{equation*}
		\begin{split}
			\operatorname{Tr}_g I \leq& \int_{M} F^{\prime}\left(\frac{|u^*h|^2}{4}\right)\lambda\Phi_{n,F}(\delta)|du |^2 \d v_g.
					\end{split}
	\end{equation*}
In the case that $ F^{\prime\prime}=F^{\prime},$ 
	\begin{equation*}
		\begin{split}
			\operatorname{Tr}_g I \leq& \int_{M} F^{\prime}\left(\frac{|u^*h|^2}{4}\right)\lambda\bigg[\left(\lambda^2m+2m+1\right)\left\{\frac{n}{4} k_{3}^{2}(\delta)+k_{3}(\delta)+1\right\}-\frac{2 \delta}{1+\delta}(n-1)\bigg]|du |^2 \d v_g.
		\end{split}
	\end{equation*}
	
\end{proof}
%
%

Next, we show the Howard type theorem for $F$-symphonic map with potential. In \cite{MR1145657}, Takeuchi obtained Howard type theroem for $p$ harmonic map. Later, in \cite{ara2001stability}, Ara generalized Takeuchi's result to $F$ harmonic map. 
	\begin{thm}
		Let $F:[0, \infty) \rightarrow[0, \infty)$ be a $C^{2}$ strictly increasing function, $\nabla^2 H \leq 0$ . Let $N$ be a compact simply-connected $\delta$-pinched $n$-dimensional Riemannian manifold. Assume that $n$ and $\delta$ satisfy $n>\frac{4}{m\lambda}c_F+\lambda+2m+1$ and
		\[
		\Psi_{n, F}(\delta):=\int_{0}^{\pi}\left\{\left(\frac{4}{m\lambda}c_F+\lambda+2m\right) g_{2}(t, \delta)\left(\frac{\sin \sqrt{\delta} t}{\sqrt{\delta}}\right)^{n-1}-(n-1) \delta \cos ^{2}(t) \sin ^{n-1}(t)\right\} d t<0,
		\]
		where $g_{2}(t, \delta)=\max \left\{\cos ^{2}(t), \delta \sin ^{2}(t) \cot ^{2}(\sqrt{\delta} t)\right\}$. Then for any compact Riemannian manifold $M$, 	every  $\Phi_{sym}$-stable conformal $ F $-symphonic map with potential $u: M \rightarrow N$ with conformal factor $ \lambda $  is constant.
	\end{thm}
	
	\begin{proof} 
		We follow \cite{ara2001stability}, let $ V^{y}=\nabla f\circ \rho_y,$  we can approximate $ f $ by $ f_k $, one can refer to \cite{ara2001stability} for the construction of $ f_k $.  Let   $ V_k^{y}=\nabla f_k\circ \rho_y  $. Then $ V_k^y $ converges uniformly to $ V^y $ as $ k\to \infty.$ By the second variation formula for $ F $-symphonic map with potential, we have 
		\begin{equation*}
			\begin{split}
				I(V^y,V^y)
				\leq 
				&\int_{M} F^{\prime \prime}\left(\frac{\left\|u^{*} h\right\|^{2}}{4}\right) \left\langle \widetilde{\nabla} V^y, \sigma_{u}\right\rangle \left\langle \widetilde{\nabla} V^y, \sigma_{u}\right\rangle d v_{g} \\
				&	+\int_{M} F^{\prime}\left(\frac{\left\|u^{*} h\right\|^{2}}{4}\right) \sum_{i, j=1}^{m} h\left(\widetilde{\nabla}_{e_{i}} V^y, \widetilde{\nabla}_{e_{j}} V^y\right) h\left(d u\left(e_{i}\right), d u\left(e_{j}\right)\right) d v_{g} \\
				&	+\int_{M} F^{\prime}\left(\frac{\left\|u^{*} h\right\|^{2}}{4}\right) \sum_{i, j=1}^{m} h\left(\widetilde{\nabla}_{e_{i}} V^y, d u\left(e_{j}\right)\right) h\left(\widetilde{\nabla}_{e_{i}} V^y, d u\left(e_{j}\right)\right) d v_{g} \\
				&	+\int_{M} F^{\prime}\left(\frac{\left\|u^{*} h\right\|^{2}}{4}\right) \sum_{i, j=1}^{m} h\left(\widetilde{\nabla}_{e_{i}} V^y, d u\left(e_{j}\right)\right) h\left(d u\left(e_{i}\right), \widetilde{\nabla}_{e_{j}} V^y\right) d v_{g} \\
				&	+\int_{M} F^{\prime}\left(\frac{\left\|u^{*} h\right\|^{2}}{4}\right) \sum_{i, }^{m} h\left(R^{N}\left(V^y, d u\left(e_{i}\right)\right) V^y, d u\left(e_{i}\right)\right) \lambda d v_{g},\\
						\end{split}
		\end{equation*}

	Notice that 
	\begin{equation*}
		\begin{split}
			\left\langle \widetilde{\nabla} V^y, \sigma_{u}\right\rangle \left\langle \widetilde{\nabla} V^y, \sigma_{u}\right\rangle=\lambda |\widetilde{\nabla}_{e_i} V^y |^2.
		\end{split}
	\end{equation*}

	Moreover, by \cite[(3.7)]{ara2001stability}, we have 
\begin{equation*}
	\begin{split}
		\sum_{i, j=1}^{m} h\left(\widetilde{\nabla}_{e_{i}} V^y, d u\left(e_{j}\right)\right) h\left(d u\left(e_{i}\right), \widetilde{\nabla}_{e_{j}} V^y\right)\leq g_2(\rho,\delta)|du|^4.
	\end{split}
\end{equation*}
Thus, we have 
\begin{equation*}
	\begin{split}
			I(V^y,V^y)\leq & 	\int_{M} F^{\prime}\left(\frac{\left\|u^{*} h\right\|^{2}}{4}\right) \sum_{i=1 }^{m}   \bigg[\frac{4}{m\lambda}c_F|\tilde{\nabla}_{e_i} V^y|^2+\lambda|\tilde{\nabla}_{e_i} V^y|^2\\
		&+2g_2(\rho,\delta)|du|^4+\lambda h\left( R^{N}\left(V^y, d u\left(e_{i}\right)\right) V^y, d u\left(e_{i}\right)\right)\bigg] d v_{g}.
	\end{split}
\end{equation*}
	Taking limits, we have 	
		\begin{equation*}
			\begin{split}
				&\lim\limits_{k\to \infty}	\int_N  I(V_{k}^{y},V_{k}^{y})dv_N(y)\\
				\leq & \int_N	\int_{M} F^{\prime}\left(\frac{\left\|u^{*} h\right\|^{2}}{4}\right) \sum_{i, }^{m}  \bigg[(\frac{4}{m\lambda}c_F+\lambda)|\tilde{\nabla}_{e_i} V^y|^2+\lambda h\left( R^{N}\left(V^y, d u\left(e_{i}\right)\right) V^y, d u\left(e_{i}\right)\right)\\
				&+2g_2(\rho,\delta)|du|^4\bigg] d v_{g} dv_N(y).
			\end{split}
		\end{equation*}
By \cite[(3.3),(3.4)]{ara2001stability}, we have 
		\begin{equation*}
			\begin{split}
				&\lim\limits_{k\to \infty}	\int_N  I(V_{k}^{y},V_{k}^{y})dv_N(y)\\
							\leq & 	\int_{M} F^{\prime}\left(\frac{\left\|u^{*} h\right\|^{2}}{4}\right)\bigg( \bigg[  (\frac{4}{m\lambda}c_F+\lambda)g_2(\rho,\delta) |du |^2Vol(S^{n-1})\left(\frac{\sin(\sqrt{\delta}t)}{\sqrt{\delta}} \right)^{n-1}\\
				&\quad\quad\quad\quad\quad\quad\quad\quad -\lambda(n-1)|du |^2Vol(S^{n-1}) \int_{0}^{\pi}\delta\cos^2(\rho)\sin^{n-1}(\rho)d\rho\bigg]\\
				&\quad\quad\quad\quad\quad\quad\quad\quad+2g_2(\rho,\delta)|du|^4 \left(\frac{\sin(\sqrt{\delta}t)}{\sqrt{\delta}} \right)^{n-1}Vol(S^{n-1})\bigg)d v_{g} .
			\end{split}
		\end{equation*}
		using the fact that $ |du |^2=\lambda m,$  by the same argument in \cite{{ara2001stability}}, we can get a contradiction.
	\end{proof}

Motivated by the method in \cite{torbaghan2022stability}, we can establish  
\begin{thm}\label{thm2.1}
	Let $u:\left(M^{m}, g\right) \longrightarrow\left(N^{n}, h\right)$ be a nonconstant conformal $F$-symphonic map with potential $ H $. Suppose that $\nabla^2 H \leq 0$  and the conformal factor of $ u $ is $ \lambda $.  Let $ \theta(v)=\sum_{a=1}^n |B(v,v_a) |^2 $, $h(x)=\max\{|B(v,v) |^2: v\in T_xN, |v|=1\}. $   For each $x \in N$ and any unit vector $v \in T_{x} N$, Ricci tensor of $ N $ satisfies 	
	\[
		\frac{8m}{\lambda}c_F  h+|A|^2+\theta(v)<\operatorname{Ric}(v, v).
	\]
 Then, $u$ is $ \Phi_{sym} $-unstable.
\end{thm}
\begin{proof} Here we modify the proof of \cite[Theorem 10]{torbaghan2022stability}. Let $ \{e_i\}_{i=1}^{m} $ be a local orthonormal frame field
	on $ M $ and $ \omega $ be a parallel vector field in $ R^{n+p} $. We take $V=\omega^\top $ in the index form of $ F $-symphonic map with potential (c.f.  \cite{ara2001stability}), 
	\begin{equation}
		\begin{split}
		I(\omega^\top,\omega^\top)=&-\int_{M} \operatorname{Hess}H(\omega^\top, \omega^\top) d v_{g} \\
			&+\int_{M} F^{\prime \prime}\left(\frac{\left\|u^{*} h\right\|^{2}}{4}\right)\left\langle \widetilde{\nabla} \omega^\top, \sigma_{u}\right\rangle \left\langle \widetilde{\nabla} \omega^\top, \sigma_{u}\right\rangle d v_{g} \\
			&	+\int_{M} F^{\prime}\left(\frac{\left\|u^{*} h\right\|^{2}}{4}\right) \sum_{i, j=1}^{m} h\left(\widetilde{\nabla}_{e_{i}} \omega^\top, \widetilde{\nabla}_{e_{j}} \omega^\top\right) h\left(d u\left(e_{i}\right), d u\left(e_{j}\right)\right) d v_{g} \\
			&	+\int_{M} F^{\prime}\left(\frac{\left\|u^{*} h\right\|^{2}}{4}\right) \sum_{i, j=1}^{m} h\left(\widetilde{\nabla}_{e_{i}} \omega^\top, d u\left(e_{j}\right)\right) h\left(\widetilde{\nabla}_{e_{i}} \omega^\top, d u\left(e_{j}\right)\right) d v_{g} \\
			&	+\int_{M} F^{\prime}\left(\frac{\left\|u^{*} h\right\|^{2}}{4}\right) \sum_{i, j=1}^{m} h\left(\widetilde{\nabla}_{e_{i}} \omega^\top, d u\left(e_{j}\right)\right) h\left(d u\left(e_{i}\right), \widetilde{\nabla}_{e_{j}} \omega^\top\right) d v_{g} \\
			&	+\int_{M} F^{\prime}\left(\frac{\left\|u^{*} h\right\|^{2}}{4}\right) \sum_{i, j=1}^{m} h\left(R^{N}\left(\omega^\top, d u\left(e_{i}\right)\right) \omega^\top, d u\left(e_{j}\right)\right) h\left(d u\left(e_{i}\right), d u\left(e_{j}\right)\right) d v_{g},
		\end{split}
	\end{equation}

So, by \cite[(30)(31)]{torbaghan2022stability}, we can get 
\begin{equation}\label{pok}
	\begin{split}
		h\left(\widetilde{\nabla}_{e_{i}} \omega^\top, d u\left(e_{j}\right)\right)= \left \langle B(du(e_i),du(e_j)), \omega^\perp \right \rangle 
	\end{split}
\end{equation}
Thu, we have 
\begin{equation}\label{}
	\begin{split}
		I(\omega^\top,\omega^\top)\leq &\jifen{F^{\prime\prime}\left(\frac{\left\|u^{*} h\right\|^{2}}{4}\right) \left \langle B(du(e_i),du(e_i)), \omega^\perp \right \rangle^2 \lambda^2 }\\
		&+\int_{M} F^{\prime}\left(\frac{\left\|u^{*} h\right\|^{2}}{4}\right)\bigg\{\lambda|A^{\omega^\perp}(du(e_i)) |^2\\
		&+F^{\prime}\left(\frac{\left\|u^{*} h\right\|^{2}}{4}\right) 2\left \langle B(du(e_i),du(e_j)), \omega^\perp \right \rangle^2  \d v_g+\sum_{t=1}^{m} h( R^N\left(\omega^\top, u_* e_{i}\right) \omega^\top,u_* e_{i} )\bigg\} \d v_{g}\\
	\end{split}
\end{equation}
Thus, set $ v_i=\frac{du(e_i)}{|du(e_i)|}, $ as in \cite[Theorem 10]{torbaghan2022stability}		
\begin{equation*}
	\begin{split}
		\operatorname{Tr}_g I\leq &\jifen{F^{\prime\prime}\left(\frac{\left\|u^{*} h\right\|^{2}}{4}\right) |B(du(e_i),du(e_i))|^2  }+ \int_M  F^{\prime }(\frac{|d u|^{2}}{2}) |du|^4 |A|^2 \d v_g\\
		&+\int_{M} F^{\prime}\left(\frac{\left\|u^{*} h\right\|^{2}}{4}\right) \cdot\left\{|\theta(v_i) -Ric(v_i,v_i)\right\}\d v_{g}\\
		\leq &\jifen{F^{\prime\prime}\left(\frac{\left\|u^{*} h\right\|^{2}}{4}\right) |du |^4 h(x)  }+\int_M  F^{\prime }(\frac{|d u|^{2}}{2}) |du|^4 |A|^2 \d v_g\\
		&+\int_{M} F^{\prime }(\frac{|d u|^{2}}{2})\sum_{i}|du(e_i) |^2 \cdot\left\{|\theta(v_i) -Ric(v_i,v_i)\right\} \d v_{g}\\
	\end{split}
\end{equation*}
By the assumption, we get 
\[
2\frac{F^{\prime\prime}\left(\frac{\left\|u^{*}h\right\|^{2}}{4}\right)}{F^{\prime}\left(\frac{\left\|u^{*}h\right\|^{2}}{4}\right)}|du|^2 h+|A|^2+\theta(v)<\operatorname{Ric}(v, v).
\]
\end{proof}

\begin{thm}\label{thm2.2} 	Let $u:\left(M^{n-1}, g\right) \longrightarrow\left(N^{n}, h\right)$ be a nonconstant $F$-symphonic map with potential $ H $. If u is totally gedesic immersion, then $ u $ is  $\Phi_{sym} $-unstable if the Ricci curvature
	of  $ N^n $ is positive.
\end{thm}

\begin{proof}
Here we modify the proof of \cite[Theorem 10]{torbaghan2022stability}. Let $ V $ be a unit normal vector field of $ M^{n-1}  $ in
$ N^n $.	Let $ I(V,V) $  be the index form of $F$-symphonic map with potential $ H $. 
	Since $ u $ is the totally geodesic, then  by \cite[(40)]{torbaghan2022stability}			
	\begin{equation*}
		\begin{split}
			\tilde{\nabla} V=0.
		\end{split}
	\end{equation*}
	So, we have 
	\begin{equation*}
		\begin{split}
			I(V,V)=-\int_{M}  F^{\prime}\left(\frac{\left\|u^{*}h\right\|^{2}}{4}\right) Ric(V,V) \lambda \d v_g.
		\end{split}
	\end{equation*}
\end{proof}

　\bibliographystyle{plain}
\bibliography{mybib2022}

\end{document}